\documentclass[10pt, reqno]{amsart}
\usepackage{amssymb,amsmath,amsthm,latexsym,booktabs,todonotes, color, comment, eucal,uncial}
\usepackage[dvipsnames]{xcolor}
\usepackage[normalem]{ulem}
\usepackage{tikz}
\theoremstyle{definition}
\newtheorem{definition}{Definition}

\newtheorem{problem}{Problem}

\newtheorem{example}[definition]{Example}

\theoremstyle{plain}
\newtheorem{lemma}[definition]{Lemma}
\newtheorem{proposition}[definition]{Proposition}
\newtheorem{theorem}[definition]{Theorem}

\usepackage{amsfonts}
\usepackage{graphicx}
\usepackage{url}
\usepackage{tikz}
\definecolor{brightturquoise}{rgb}{0.03, 0.91, 0.87}
\definecolor{brightpink}{rgb}{1.0, 0.0, 0.5}
\definecolor{brilliantrose}{rgb}{1.0, 0.33, 0.64}
\definecolor{capri}{rgb}{0.0, 0.75, 1.0}
\definecolor{blue-green}{rgb}{0.0, 0.87, 0.87}
\definecolor{bleudefrance}{rgb}{0.19, 0.55, 0.91}

\begin{document}

\title[$C_4$-face-magic labelings on a $4 \times 4$ Klein 
bottle grid graph]{$C_4$-face-magic labeling on a 
$4 \times 4$ Klein \\
bottle grid graph}
\author{Timothy Myers}
\address{Department of Mathematics,
         Howard University,
         Washington, DC 20059
         USA}
\email{timyers@howard.edu}
\author{Stephen J. Curran}
\address{Department of Mathematics,
         University of Pittsburgh at Johnstown,
         450 Schoolhouse Rd,
         Johnstown, PA 15904
         USA}
         \email{sjcurran@pitt.edu}

\date{\indent 2010 \textit{Mathematics Subject Classification.} 91A46, 05C78}

\begin{abstract}
For a graph $G = (V, E)$ embedded in the Klein bottle, let $\mathcal{F}(G)$ denote the set of faces of $G$.
A \textit{$C_4$-face-magic Klein bottle labeling} on $G$
is a bijection $f: V(G) \to \{1, 2, \dots, |V(G)|\}$
such that for any $F \in \mathcal{F}(G)$ with $F \cong C_4$, 
the sum of all the vertex labelings along $C_4$ is a constant. 
We say that a $C_4$-face-magic labeling
$X=\{x_{i,j} : 1\le i,j\le 4\}$ on the
$4\times 4$ Klein bottle grid graph is
\textit{horizontally (or vertically) pairwise balanced}
if $x_{2i-1,j} + x_{2i,j}=17$ for $1\le i \le 2$
and $1\le j \le 4$
(or $x_{i,2j-1} + x_{i2,j}=17$ for $1\le i \le 4$
and $1\le j \le 2$).
We show that the $4\times 4$ Klein bottle grid graph has
192 $C_4$-face-magic labelings up to symmetries on
a Klein bottle.
We classify these labelings into two categories
depending on whether a $C_4$-face-magic label preserving 
permutation of the labeling is either horizontally
pairwise balanced or vertically pairwise balanced.
These results extend known results on $C_4$-face-magic
labelings on an $m\times n$ Klein bottle grid graph.
\end{abstract}

\maketitle

\section{Introduction}

Kotzig and Rosa \cite{Kotzig} introduced graph labelings in 1970.
Graph labelings have been studied extensively since that time. 
In fact researchers \cite{SaviyoEtAl2023} have found various applications of graph labelings
to radar pulse code designs, communication network 
models, X-ray crystallography, and graph decomposition problems. 
For a thorough review of graph labelings, we direct the reader to Gallian’s dynamic
survey on graph labelings \cite{Gallian}.

In this paper we investigate $C_4$-face-magic labelings on the Klein bottle grid graph.
Let $G=(V,E)$ be a Klein bottle (planar, toroidal, projective) graph, and
let $\mathcal{F}(G)$ denote the set of faces of $G$.
Then, $G$ is called a \textit{$C_n$-face-magic 
Klein bottle, (planar, toroidal, projective)} graph if there
exists a bijection $f: V(G) \to \{1, 2, \dots, |V(G)|\}$ such that for any $F \in \mathcal{F}(G)$ with $F \cong C_n$,
the sum of all the vertex labels around $C_n$ is a constant $S$. We call the constant $S$ a \textit{$C_n$-face-magic value} of $G$.
Let $x_v =f(v)$ for all $v\in V(G)$. We shall call $\big\{x_v : v\in V(G)\big\}$ a $C_n$-face-magic Klein bottle labeling on $G$.
More generally, $C_n$-face-magic planar graph labelings are a special case of $(a, b, c)$-magic labelings
introduced by Lih \cite{Lih}.
For $a,b,c\in\{0,1\}$, an $(a, b, c)$-magic labeling is a bijective assignment of the elements of the set $\{1,2,...,a|V|+b|E|+c|F|\}$ to the 
 vertices, edges, and faces of $G=(V,E,F)$ that assigns $a$ labels to each vertex, $b$ labels to each edge, and $c$ labels to each face 
 such that the sum of labels for each face (including vertices, edges, and the face itself) is constant. 
For assorted values of $a, b$ and $c$, Baca and others \cite{Baca1, Baca2, Baca3, Kasif, Kath, Lih} have analyzed
the problem for various classes of graphs.
Wang \cite{Wang} showed that the toroidal grid graphs $C_m \times C_n$ are antimagic for all integers $m,n\geqslant 3$.
Recall that an antimagic labeling is a bijective assignment of the set $\{1,2,\ldots,|E|\}$ onto the edges
of $G$ such that, for every vertex, the sum of the labels on edges incident to the vertex is unique. 
Butt et al. \cite{Buttetal} investigated face antimagic labelings on toroidal and Klein bottle grid graphs;
they found bijective assignments of the elements from the set $\{1,2,\ldots,|V|+|E|+|F|\}$ onto $V\cup E\cup F$ such that the
induced face labels, formed by the sum of the labels on all vertices and edges incident to the face and the label on the face itself,
produce arithmetic sequences with various common differences.

Curran et al. \cite{CurranLowLocke1} examined $C_4$-face-magic toroidal labelings 
on an $m \times n$ toroidal grid graph. 
They showed that $C_m \times C_n$ admits a $C_4$-face-magic toroidal labeling if and only if either $m = 2$, or $n = 2$, or both $m$ and $n$ are even. 
Curran et al. \cite{CurranLowLocke2}  also 
investigated $C_4$-face-magic Klein bottle labelings 
on an $m \times n$ Klein bottle grid graph.
They showed that an $m \times n$ Klein bottle grid graph 
admits a $C_4$-face-magic Klein bottle labeling 
if and only if $n$ is even. 

Curran \cite{Curran2} showed that an $m \times n$ projective grid graph  admits a $C_4$-face-magic 
projective labeling if and only if either $m$ and $n$ 
have the same parity.
When $m$ and $n$ are odd, the $C_4$-face-magic projective labelings on an $m \times n$ projective grid graph
having $C_4$-face-magic value $2mn+2$ were characterized in \cite{Curran2},
and a category of the $C_4$-face-magic labelings 
having $C_4$-face-magic value $2mn+1$ or $2mn+3$
were determined in \cite{Curran3}.

Curran and Low \cite{CurranLow} showed that there are
3 antipodally balanced $C_4$-face-magic labelings on 
a $4\times 4$ toroidal grid graph up to symmetry on the torus.
Curran and Locke \cite{CurranLocke3} showed that
there are 144 $C_4$-face-magic labelings on 
a $4\times 4$ projective grid graph up to
symmetries on a projective plane.
In this paper we show that there are 192 
 $C_4$-face-magic labelings on 
a $4\times 4$ Klein bottle grid graph up to
symmetries on a Klein bottle.

\section{Preliminaries}

We begin by defining the $m\times n$ Klein bottle grid graph.
\begin{definition}
   The $m\times n$ \it Klein bottle grid graph\rm, denoted $\mathcal{K}_{m,n}$, is the graph whose vertex set is the lattice $V(\mathcal{K}_{m,n})=\big\{(i,j)\  : \ 1\le i\le m\ ,\ 1\le j\le n  \big\}$\ and whose edge set consists of the following edges. 
\begin{enumerate}
    \item[(i)]There is an edge from $(i,j)$ to $(i+1,j)$ for $1\le i\le m-1$ and $1\le j\le n$.
    \item[(ii)]There is an edge from $(i,j)$ to $(i, j+1)$ for $1\le i\le m$ and $1\le j\le n-1$.
    \item[(iii)]There is an edge from $(i,n)$ to $(i,1)$ for $1\le i\le m$.
    \item[(iv)]There is an edge from $(m,j)$ to $(1, n+1-j)$ for $1\le j\le n$.
\end{enumerate}

   The graph $\mathcal{K}_{m,n}$ has a canonical embedding in the Klein bottle. 
\end{definition}

Figure \ref{K44} illustrates the $\mathcal{K}_{4,4}$ grid graph. 

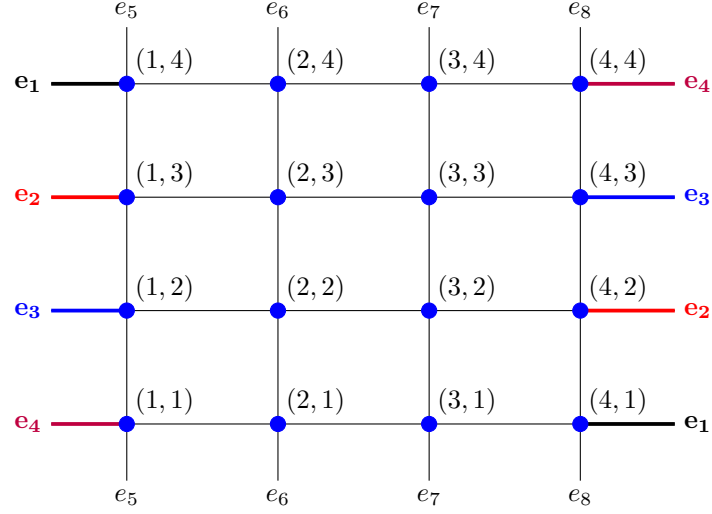
\begin{figure}
\centering
\begin{tikzpicture}
 \draw (-1,0)--(7.25,0);
 \draw (-1,1.5)--(7.25, 1.5);
 \draw (-1, 3)--(7.25, 3);
 \draw (-1, 4.5)--(7.25, 4.5);
 \draw (0,-.75)--(0, 5.25);
 \draw (2, -.75)--(2, 5.25);
 \draw (4, -.75)--(4, 5.25);
 \draw (6, -.75)--(6, 5.25);
 \draw[line width=0.5mm, color=purple] (-1,0)--(0,0);
 \draw[line width=0.5mm, color=purple] (6,4.5)--(7.25,4.5);
 \draw[line width=0.5mm, color=blue] (-1,1.5)--(0,1.5);
 \draw[line width=0.5mm, color=blue] (6,3)--(7.25,3);
 \draw[line width=0.5mm, color=red] (-1,3)--(0,3);
 \draw[line width=0.5mm, color=red] (6,1.5)--(7.25,1.5);
  \draw[line width=0.5mm, color=black] (-1,4.5)--(0,4.5);
 \draw[line width=0.5mm, color=black] (6,0)--(7.25,0);
 \node[above] at (0,5.25) {$e_5$};
 \node[below] at (0,-.75) {$e_5$};
 \node[above] at (2,5.25) {$e_6$};
 \node[below] at (2,-.75) {$e_6$};
 \node[above] at (4,5.25) {$e_7$};
 \node[below] at (4,-.75) {$e_7$};
 \node[above] at (6,5.25) {$e_8$};
 \node[below] at (6,-.75) {$e_8$};
 \node[left] at (-1,4.5) {\textcolor{black}{$\bf{e_1}$}};
 \node[right] at (7.25,0) {\textcolor{black}{$\bf{e_1}$}};
 \node[left] at (-1,3) {\textcolor{red}{$\bf{e_2}$}};
 \node[right] at (7.25,1.5) {\textcolor{red}{$\bf{e_2}$}};
 \node[left] at (-1,1.5) {\textcolor{blue}{$\bf{e_3}$}};
 \node[right] at (7.25,3) {\textcolor{blue}{$\bf{e_3}$}};
 \node[left] at (-1,0) {\textcolor{purple}{$\bf{e_4}$}};
 \node[right] at (7.25,4.5) {\textcolor{purple}{$\bf{e_4}$}};
 \draw[fill, color=blue] (0,0) circle [radius=.1];
 \draw[fill, color=blue] (0,1.5) circle [radius=.1];
 \draw[fill, color=blue] (0,3) circle [radius=.1];
 \draw[fill, color=blue] (0,4.5) circle [radius=.1];
 \draw[fill, color=blue] (2,0) circle [radius=.1];
 \draw[fill, color=blue] (2,1.5) circle [radius=.1];
 \draw[fill, color=blue] (2,3) circle [radius=.1];
 \draw[fill, color=blue] (2,4.5) circle [radius=.1];
 \draw[fill, color=blue] (4,0) circle [radius=.1];   
 \draw[fill, color=blue] (4,1.5) circle [radius=.1];   
 \draw[fill, color=blue] (4,3) circle [radius=.1];
 \draw[fill, color=blue] (4,4.5) circle [radius=.1];
 \draw[fill, color=blue] (6,0) circle [radius=.1];   
 \draw[fill, color=blue] (6,1.5) circle [radius=.1];   
 \draw[fill, color=blue] (6,3) circle [radius=.1];
 \draw[fill, color=blue] (6,4.5) circle [radius=.1];
 \node[above] at (0.5,0) {$(1,1)$};
 \node[above] at (0.5,1.5) {$(1,2)$};
 \node[above] at (0.5,3) {$(1,3)$};
 \node[above] at (0.5,4.5) {$(1,4)$};
 \node[above] at (2.5,0) {$(2,1)$};
 \node[above] at (2.5,1.5) {$(2,2)$};
 \node[above] at (2.5,3) {$(2,3)$};
 \node[above] at (2.5,4.5) {$(2,4)$};
 \node[above] at (4.5,0) {$(3,1)$};
 \node[above] at (4.5,1.5) {$(3,2)$};
 \node[above] at (4.5,3) {$(3,3)$};
 \node[above] at (4.5,4.5) {$(3,4)$};
 \node[above] at (6.5,0) {$(4,1)$};
 \node[above] at (6.5,1.5) {$(4,2)$};
 \node[above] at (6.5,3) {$(4,3)$};
 \node[above] at (6.5,4.5) {$(4,4)$};
\end{tikzpicture}
\caption{The grid graph $\mathcal{K}_{4,4}$.}\label{K44}
\end{figure}

\begin{lemma}\label{lem1}\rm(\cite[Lemma 7]{CurranLowLocke2})\it. 
    Let $m\ge 2$ and $n\ge 2$ be even integers, and suppose that the set $\{\ x_{u,v}\ : \ (u,v)\in V(\mathcal{K}_{m,n}) \}$ is a $C_4$-face-magic labeling on $\mathcal{K}_{m,n}$ with $C_4$ face-magic value $S$. Then $S=2mn+2$. In particular, the $C_4$ face magic value for $\mathcal{K}_{4,4}$ is $S=34$.
\end{lemma}

As an aid to the reader, we include a proof of Lemma \ref{lem1}. 

\begin{proof}
  By definition of a $C_4$-face-magic labeling of $\mathcal{K}_{m,n}$, 
\begin{equation*}
\{\ x_{u,v}\ : \ (u,v)\in V(\mathcal{K}_{m,n}) \}=\{1,2,...,mn\}.
\end{equation*}

  Let $m_0, n_0\in\mathbb{Z}$  such that $m=2m_0$ and $n=2n_0$. Then
\begin{equation}\label{1}
\begin{aligned}
 mnS&=4m_0n_0S=4\sum_{i=1}^{m_0}\sum_{j=1}^{n_0}\ S\\
    &=4\sum_{i=1}^{m_0}\sum_{j=1}^{n_0}\big( x_{2i-1, 2j-1}+x_{2i, 2j-1}+x_{2i-1, 2j}+x_{2i, 2j}  \big)\\
    &=4\sum_{u=1}^m\sum_{v=1}^nx_{u,v}=4\sum_{k=1}^{mn}k=4\cdot\tfrac{1}{2}\big(\ mn(mn+1)\ \big).
\end{aligned}
\end{equation}
 Thus $S=2mn+2$ by (\ref{1}).
\end{proof}

For the remainder of the paper we will focus 
exclusively on $C_4$-face-magic labelings on the
$4\times 4$ Klein bottle grid graph.
We want to characterize the $C_4$-face-magic labelings on
$\mathcal{K}_{4,4}$ up to symmetries on a Klein bottle.
In order to do this, we introduce the following two types of
$C_4$-face-magic labelings on $\mathcal{K}_{4,4}$.

\begin{definition}
Let  $X=\{x_{i,j}: 1\le i \le 4 \text{ and } 
   1\le j \le 4\}$ be a $C_4$-face-magic labeling 
   on  $\mathcal{K}_{4,4}$ with face-magic value $S$.
   \begin{enumerate}
       \item We say that $X$ is 
       \textit{vertically pairwise balanced} if
   \begin{equation*}
       x_{i,2j-1} + x_{i,2j} =\tfrac{1}{2} S
   \end{equation*}
   for all $1\le i \le 4$ and $1\le j \le 2$.
  \item Similarly, we say that $X$ is
   \textit{horizontally pairwise balanced} if
   \begin{equation*}
       x_{2i-1,j} + x_{2i,j} =\tfrac{1}{2} S
   \end{equation*}
   for all $1\le i \le 2$ and $1\le j \le 4$.
    \end{enumerate}
\end{definition}

We illustrate a vertically balanced $C_4$-face-magic labeling 
on $\mathcal{K}_{4,4}$ in Example \ref{ex:VerticallyBalanced} and a horizontally
balanced $C_4$-face-magic labeling 
in Example \ref{ex:HorizontallyBalanced}.

\begin{example}\label{ex:VerticallyBalanced}
Consider the
vertically balanced $C_4$-face-magic labeling 
on $\mathcal{K}_{4,4}$ 
with $C_4$-face-magic value $34$ 
in Figure \ref{fig:EquatorExample}.

\begin{figure}[hbt!]
\centering
\begin{tikzpicture}
 \draw (-1.5,0)--(7.75,0);
 \draw (-1.5,1.5)--(7.75, 1.5);
 \draw (-1.5, 3)--(7.75, 3);
 \draw (-1.5, 4.5)--(7.75, 4.5);
 \draw (0,-.75)--(0, 5.25);
 \draw (2, -.75)--(2, 5.25);
 \draw (4, -.75)--(4, 5.25);
 \draw (6, -.75)--(6, 5.25);
 \draw[color=red, line width=0.5mm](0,0)--(0,1.5);
 \draw[color=red, line width=0.5mm](0,3)--(0,4.5);
 \draw[color=red, line width=0.5mm](2, 0)--(2,1.5);
 \draw[color=red, line width=0.5mm](2,3)--(2,4.5);
 \draw[color=red, line width=0.5mm](4,0)--(4,1.5);
 \draw[color=red, line width=0.5mm](4,3)--(4,4.5);
  \draw[color=red, line width=0.5mm](6,0)--(6,1.5);
 \draw[color=red, line width=0.5mm](6,3)--(6,4.5);
 \node[above] at (0,5.25) {$e_5$};
 \node[below] at (0,-.75) {$e_5$};
 \node[above] at (2,5.25) {$e_6$};
 \node[below] at (2,-.75) {$e_6$};
 \node[above] at (4,5.25) {$e_7$};
 \node[below] at (4,-.75) {$e_7$};
 \node[above] at (6,5.25) {$e_8$};
 \node[below] at (6,-.75) {$e_8$};
 \node[left] at (-1.5,4.5) {\textcolor{black}{$e_1$}};
 \node[right] at (7.75,0) {\textcolor{black}{$e_1$}};
 \node[left] at (-1.5,3) {\textcolor{red}{$e_2$}};
 \node[right] at (7.75,1.5) {\textcolor{red}{$e_2$}};
 \node[left] at (-1.5,1.5) {\textcolor{blue}{$e_3$}};
 \node[right] at (7.75,3) {\textcolor{blue}{$e_3$}};
 \node[left] at (-1.5,0) {\textcolor{purple}{$e_4$}};
 \node[right] at (7.75,4.5) {\textcolor{purple}{$e_4$}};
 \draw[fill, color=red] (0,0) circle [radius=.1];
 \draw[fill, color=red] (0,1.5) circle [radius=.1];
 \draw[fill, color=red] (0,3) circle [radius=.1];
 \draw[fill, color=red] (0,4.5) circle [radius=.1];
 \draw[fill, color=red] (2,0) circle [radius=.1];
 \draw[fill, color=red] (2,1.5) circle [radius=.1];
 \draw[fill, color=red] (2,3) circle [radius=.1];
 \draw[fill, color=red] (2,4.5) circle [radius=.1];
 \draw[fill, color=red] (4,0) circle [radius=.1];   
 \draw[fill, color=red] (4,1.5) circle [radius=.1];   
 \draw[fill, color=red] (4,3) circle [radius=.1];
 \draw[fill, color=red] (4,4.5) circle [radius=.1];
 \draw[fill, color=red] (6,0) circle [radius=.1];   
 \draw[fill, color=red] (6,1.5) circle [radius=.1];   
 \draw[fill, color=red] (6,3) circle [radius=.1];
 \draw[fill, color=red] (6,4.5) circle [radius=.1];
 \node[above] at (0.7,0) {$x_{1,1}=\textcolor{red}{\bf{1}}$};
 \node[above] at (0.7,1.5) {$x_{1,2}=\textcolor{brilliantrose}{{16}}$};
 \node[above] at (0.7,3) {$x_{1,3}=\textcolor{brilliantrose}{3}$};
 \node[above] at (0.7,4.5) {$x_{1,4}=\textcolor{red}{\bf{14}}$};
 \node[above] at (2.8,0) {$x_{2,1}=\textcolor{red}{\bf{8}}$};
 \node[above] at (2.8,1.5) {$x_{2,2}=\textcolor{brilliantrose}{9}$};
 \node[above] at (2.8,3) {$x_{2,3}=\textcolor{brilliantrose}{6}$};
 \node[above] at (2.8,4.5) {$x_{2,4}=\textcolor{red}{\bf{11}}$};
 \node[above] at (4.7,0) {$x_{3,1}=\textcolor{red}{\bf{2}}$};
 \node[above] at (4.7,1.5) {$x_{3,2}=\textcolor{brilliantrose}{15}$};
 \node[above] at (4.7,3) {$x_{3,3}=\textcolor{brilliantrose}{4}$};
 \node[above] at (4.7,4.5) {$x_{3,4}=\textcolor{red}{\bf{13}}$};
 \node[above] at (6.8,0) {$x_{4,1}=\textcolor{red}{\bf{7}}$};
 \node[above] at (6.8,1.5) {$x_{4,2}=\textcolor{brilliantrose}{{10}}$};
 \node[above] at (6.8,3) {$x_{4,3}=\textcolor{brilliantrose}{{5}}$};
 \node[above] at (6.8,4.5) {$x_{4,4}=\textcolor{red}{\bf{12}}$};
\end{tikzpicture}
\caption{$C_4$-face-magic vertically pairwise balanced 
labeling on $\mathcal{K}_{4,4}$.} 
\label{fig:EquatorExample}
\end{figure}
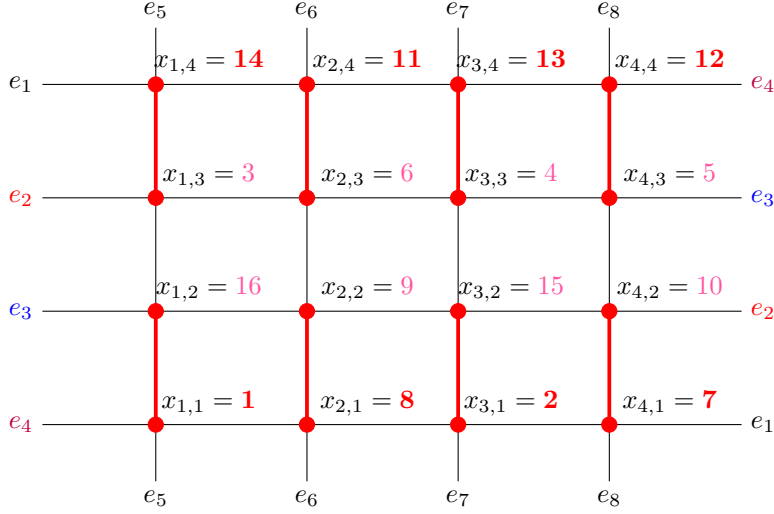
\end{example}

\begin{example}\label{ex:HorizontallyBalanced}
Consider the 
horizontally pairwise balanced $C_4$-face-magic labeling 
on $\mathcal{K}_{4,4}$ with $C_4$-face-magic value $34$ in Figure \ref{fig:PairVerticalExample}.

\begin{figure}[hbt!]
\centering
\begin{tikzpicture}
 \draw (-1.5,0)--(7.75,0);
 \draw (-1.5,1.5)--(7.75, 1.5);
 \draw (-1.5, 3)--(7.75, 3);
 \draw (-1.5, 4.5)--(7.75, 4.5);
 \draw (0,-.75)--(0, 5.25);
 \draw (2, -.75)--(2, 5.25);
 \draw (4, -.75)--(4, 5.25);
 \draw (6, -.75)--(6, 5.25);
 \draw[color=blue, line width=0.5mm](0,0)--(2,0);
 \draw[color=blue, line width=0.5mm](4,0)--(6,0);
 \draw[color=blue, line width=0.5mm](0,1.5)--(2,1.5);
 \draw[color=blue, line width=0.5mm](4,1.5)--(6,1.5);
 \draw[color=blue, line width=0.5mm](0,3)--(2,3);
 \draw[color=blue, line width=0.5mm](4,3)--(6,3);
 \draw[color=blue, line width=0.5mm](0,4.5)--(2,4.5);
 \draw[color=blue, line width=0.5mm](4,4.5)--(6,4.5);
 \node[above] at (0,5.25) {$e_5$};
 \node[below] at (0,-.75) {$e_5$};
 \node[above] at (2,5.25) {$e_6$};
 \node[below] at (2,-.75) {$e_6$};
 \node[above] at (4,5.25) {$e_7$};
 \node[below] at (4,-.75) {$e_7$};
 \node[above] at (6,5.25) {$e_8$};
 \node[below] at (6,-.75) {$e_8$};
 \node[left] at (-1.5,4.5) {\textcolor{black}{$e_1$}};
 \node[right] at (7.75,0) {\textcolor{black}{$e_1$}};
 \node[left] at (-1.5,3) {\textcolor{red}{$e_2$}};
 \node[right] at (7.75,1.5) {\textcolor{red}{$e_2$}};
 \node[left] at (-1.5,1.5) {\textcolor{blue}{$e_3$}};
 \node[right] at (7.75,3) {\textcolor{blue}{$e_3$}};
 \node[left] at (-1.5,0) {\textcolor{purple}{$e_4$}};
 \node[right] at (7.75,4.5) {\textcolor{purple}{$e_4$}};
 \draw[fill, color=blue] (0,0) circle [radius=.1];
 \draw[fill, color=blue] (0,1.5) circle [radius=.1];
 \draw[fill, color=blue] (0,3) circle [radius=.1];
 \draw[fill, color=blue] (0,4.5) circle [radius=.1];
 \draw[fill, color=blue] (2,0) circle [radius=.1];
 \draw[fill, color=blue] (2,1.5) circle [radius=.1];
 \draw[fill, color=blue] (2,3) circle [radius=.1];
 \draw[fill, color=blue] (2,4.5) circle [radius=.1];
 \draw[fill, color=blue] (4,0) circle [radius=.1];   
 \draw[fill, color=blue](4,1.5) circle [radius=.1];   
 \draw[fill, color=blue] (4,3) circle [radius=.1];
 \draw[fill, color=blue] (4,4.5) circle [radius=.1];
 \draw[fill, color=blue] (6,0) circle [radius=.1];   
 \draw[fill, color=blue] (6,1.5) circle [radius=.1];   
 \draw[fill, color=blue] (6,3) circle [radius=.1];
 \draw[fill, color=blue] (6,4.5) circle [radius=.1];
 \node[above] at (0.7,0) {$x_{1,1}=\textcolor{blue}{\bf{1}}$};
 \node[above] at (0.7,1.5) {$x_{1,2}=\textcolor{blue}{\bf{8}}$};
 \node[above] at (0.7,3) {$x_{1,3}=\textcolor{blue}{\bf{2}}$};
 \node[above] at (0.7,4.5) {$x_{1,4}=\textcolor{blue}{\bf{7}}$};
 \node[above] at (2.8,0) {$x_{2,1}=\textcolor{bleudefrance}{16}$};
 \node[above] at (2.8,1.5) {$x_{2,2}=\textcolor{bleudefrance}{9}$};
 \node[above] at (2.8,3) {$x_{2,3}=\textcolor{bleudefrance}{15}$};
 \node[above] at (2.8,4.5) {$x_{2,4}=\textcolor{bleudefrance}{10}$};
 \node[above] at (4.7,0) {$x_{3,1}=\textcolor{bleudefrance}{3}$};
 \node[above] at (4.7,1.5) {$x_{3,2}=\textcolor{bleudefrance}{6}$};
 \node[above] at (4.7,3) {$x_{3,3}=\textcolor{bleudefrance}{4}$};
 \node[above] at (4.7,4.5) {$x_{3,4}=\textcolor{bleudefrance}{5}$};
 \node[above] at (6.8,0) {$x_{4,1}=\textcolor{blue}{\bf{14}}$};
 \node[above] at (6.8,1.5) {$x_{4,2}=\textcolor{blue}{\bf{11}}$};
 \node[above] at (6.8,3) {$x_{4,3}=\textcolor{blue}{\bf{13}}$};
 \node[above] at (6.8,4.5) {$x_{4,4}=\textcolor{blue}{\bf{12}}$};
\end{tikzpicture}
\caption{$C_4$-face-magic horizontally pairwise balanced 
labeling on $\mathcal{K}_{4,4}$.} 
\label{fig:PairVerticalExample}
\end{figure}
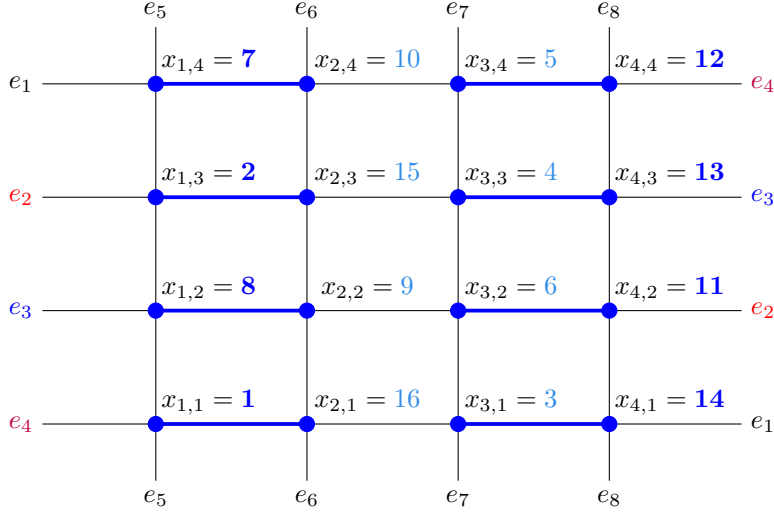
\end{example}

\begin{definition}
Let $X=\{ x_{i,j} : (i,j) \in V(\mathcal{K}_{4,4})\}$ 
and  $Y=\{ y_{i,j} : (i,j) \in V(\mathcal{K}_{4,4})\}$ 
be $C_4$-face magic labelings on $\mathcal{K}_{4,4}$. 
We say that $X$ is \textit{Klein bottle labeling 
equivalent} to $Y$ if there exists a graph isomorphism 
$\varphi$ induced by a homeomorphism on the Klein bottle 
grid such that $x_{i,j}=y_{\varphi(i,j)}$ for all 
$(i,j)\in V(\mathcal{K}_{4,4})$.
\end{definition}

In particular, the four symmetries given in Definition
\ref{defn:KBSymmetries} generate all the 
graph isomorphisms on $\mathcal{K}_{4,4}$ that are 
induced by a homeomorphism on the Klein bottle.

\begin{definition}\label{defn:KBSymmetries}
We define the following graph isomorphisms of $\mathcal{K}_{4,4}$ 
induced by a symmetry on the Klein bottle:
\begin{align}
  E(i,j)&=
  \begin{cases}
   (i+1,j) & \text{if $1\le i\le 3, 1\le j\le 4$,}\\
   (1,5-j) & \text{if $i=4, 1\le j\le 4$,}
  \end{cases}\label{te}\\
    &\notag\\
  H(i,j)&=E^4(i,j)=(i, 5-j)\ \text{for $1\le i\le 4, 1\le j\le 4$,}\ \label{th} \\
  &\notag\\
  N(i,j)&=\big(\ i\ ,\ (j+2)\ \text{mod}\ 4\  \big)\ \text{for $1\le i\le 4, 1\le j\le 4$, and}\label{tn}\\
  &\notag\\
  V(i,j)&=
   \begin{cases}
       \big(i\ ,\ j\big)\ & \text{if $i=1, 1\le j\le 4$,}\\
       \big(\ (6-i)\ \text{mod}\ 4\ , \ 5-j\  \big)\ & \text{if $2\le i\le 4, 1\le j\le 4$.}
   \end{cases}\label{tv}
\end{align}
We define the \textit{Klein bottle labeling symmetry 
group on $\mathcal{K}_{4,4}$} as 
\begin{equation}\label{kbls}
\begin{aligned}
\mathit{KBLS}(4)&=\langle\ E, H, N, V \ \rangle\ =\ \big\{ E^i N^j V^k \ : \ 0\le i\le 7, \ j,k\in\{0,1\} \big\}\\
&=\big\{ E^i H^j N^k V^{\ell} \ : \ 0\le i\le 3,\ i,j,\ell\in\{0,1\} \big\}
\end{aligned}
\end{equation}
\end{definition}

\begin{example}
    The result of the symmetry given by $V$ in Definition \ref{defn:KBSymmetries} is illustrated
    in Figure \ref{fig:VTransform}.

\begin{figure}
\centering

\begin{tikzpicture}
 \draw (-1,0)--(7.25,0);
 \draw (-1,1.5)--(7.25, 1.5);
 \draw (-1, 3)--(7.25, 3);
 \draw (-1, 4.5)--(7.25, 4.5);
 \draw (0,-.75)--(0, 5.25);
 \draw (2, -.75)--(2, 5.25);
 \draw (4, -.75)--(4, 5.25);
 \draw (6, -.75)--(6, 5.25);
 \node[above] at (0,5.25) {$e'_5$};
 \node[below] at (0,-.75) {$e'_5$};
 \node[above] at (2,5.25) {$e'_6$};
 \node[below] at (2,-.75) {$e'_6$};
 \node[above] at (4,5.25) {$e'_7$};
 \node[below] at (4,-.75) {$e'_7$};
 \node[above] at (6,5.25) {$e'_8$};
 \node[below] at (6,-.75) {$e'_8$};
 \node[left] at (-1,4.5) {\textcolor{black}{$e'_1$}};
 \node[right] at (7.25,0) {\textcolor{black}{$e'_1$}};
 \node[left] at (-1,3) {\textcolor{red}{$e'_2$}};
 \node[right] at (7.25,1.5) {\textcolor{red}{$e'_2$}};
 \node[left] at (-1,1.5) {\textcolor{blue}{$e'_3$}};
 \node[right] at (7.25,3) {\textcolor{blue}{$e'_3$}};
 \node[left] at (-1,0) {\textcolor{purple}{$e'_4$}};
 \node[right] at (7.25,4.5) {\textcolor{purple}{$e'_4$}};
 \draw[fill, color=blue] (0,0) circle [radius=.1];
 \draw[fill, color=blue] (0,1.5) circle [radius=.1];
 \draw[fill, color=blue] (0,3) circle [radius=.1];
 \draw[fill, color=blue] (0,4.5) circle [radius=.1];
 \draw[fill, color=red] (2,0) circle [radius=.1];
 \draw[fill, color=red] (2,1.5) circle [radius=.1];
 \draw[fill, color=red] (2,3) circle [radius=.1];
 \draw[fill, color=red] (2,4.5) circle [radius=.1];
 \draw[fill, color=blue] (4,0) circle [radius=.1];   
 \draw[fill, color=blue] (4,1.5) circle [radius=.1];   
 \draw[fill, color=blue] (4,3) circle [radius=.1];
 \draw[fill, color=blue] (4,4.5) circle [radius=.1];
 \draw[fill, color=red] (6,0) circle [radius=.1];   
 \draw[fill, color=red] (6,1.5) circle [radius=.1];   
 \draw[fill, color=red] (6,3) circle [radius=.1];
 \draw[fill, color=red] (6,4.5) circle [radius=.1];
 \node[above] at (0.5,0) {$(1,1)$};
 \node[above] at (0.5,1.5) {$(1,2)$};
 \node[above] at (0.5,3) {$(1,3)$};
 \node[above] at (0.5,4.5) {$(1,4)$};
 \node[above] at (2.5,0) {$(4,\textcolor{red}{4})$};
 \node[above] at (2.5,1.5) {$(4,\textcolor{red}{3})$};
 \node[above] at (2.5,3) {$(4,\textcolor{red}{2})$};
 \node[above] at (2.5,4.5) {$(4,\textcolor{red}{1})$};
 \node[above] at (4.5,0) {$(3,\textcolor{red}{4})$};
 \node[above] at (4.5,1.5) {$(3,\textcolor{red}{3})$};
 \node[above] at (4.5,3) {$(3,\textcolor{red}{2})$};
 \node[above] at (4.5,4.5) {$(3,\textcolor{red}{1})$};
 \node[above] at (6.5,0) {$(2,\textcolor{red}{4})$};
 \node[above] at (6.5,1.5) {$(2,\textcolor{red}{3})$};
 \node[above] at (6.5,3) {$(2,\textcolor{red}{2})$};
 \node[above] at (6.5,4.5) {$(2,\textcolor{red}{1})$};
\end{tikzpicture}
\caption{The transform $V$ on the grid graph $\mathcal{K}_{4,4}$.}\label{fig:VTransform}
\end{figure}
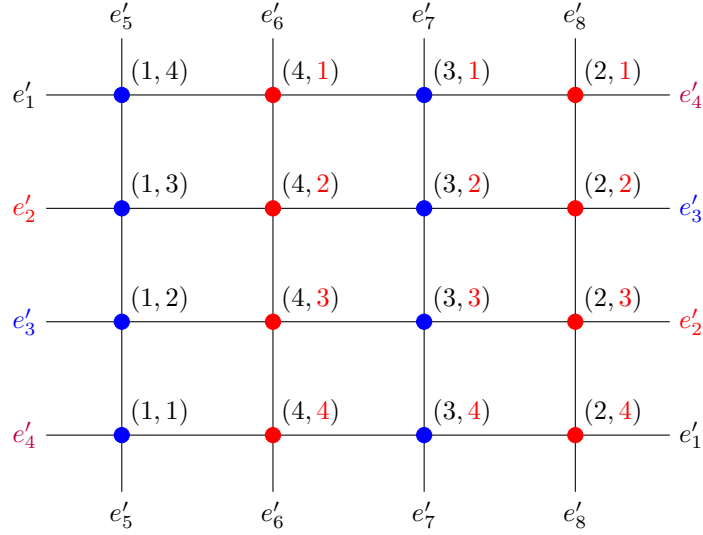

\end{example}

When we incorporate the symmetries of $\mathit{KBLS}(4)$
on a $C_4$-face-magic labeling on $\mathcal{K}_{4,4}$,
we observe that a $C_4$-face-magic labeling
on $\mathcal{K}_{4,4}$ is Klein bottle labeling equivalent
to a unique $C_4$-face-magic labeling
on $\mathcal{K}_{4,4}$ that satisfies the two properties 
listed in Proposition \ref{prop:LabelEquiv}.

\begin{proposition}\label{prop:LabelEquiv}
    Let  $X=\{x_{i,j}: 1\le i \le 4 \text{ and } 
   1\le j \le 4\}$ be a $C_4$-face-magic labeling 
   on  $\mathcal{K}_{4,4}$ with face-magic value $S=34$.
   Then there exists a unique 
   $C_4$-face-magic labeling 
   $Y=\{y_{i,j}: 1\le i \le 4 \text{ and } 
   1\le j \le 4\}$ 
   on  $\mathcal{K}_{4,4}$
  which is Klein bottle labeling equivalent to $X$ such that
   \begin{enumerate}
       \item $y_{1,1}=1$ and
        \item $\min\{ y_{2,2},y_{2,3}\}
        <\min\{ y_{4,2},y_{4,3}\}$.
   \end{enumerate}
\end{proposition}

\begin{proof}
   Suppose $x_{i,j}=1$. If $j\in \{1,4\}$, let $Z=E^{9-i}(X)$ if $j=1$ and
   let $Z=E^{5-i}(X)$ if $j=4$. Then $z_{1,1}=1$.
   If $j\in \{2,3\}$, let $Z=E^{9-i}(X)$ if $j=2$ and
   let $Z=E^{5-i}(X)$ if $j=3$. Then $z_{1,3}=1$.

   Next, let $W=Z$ if $z_{1,1}=1$, and
        let $W=N(Z)$ if $z_{1,3}=1$. 
        Then  $w_{1,1}=1$.

   Further, let $Y=W$ if $\min\{ w_{2,2},w_{2,3}\}
        <\min\{ w_{4,2},w_{4,3}\}$, and
        let $Y=V(W)$ if $\min\{ w_{2,2},w_{2,3}\}
        >\min\{ w_{4,2},w_{4,3}\}$. 
        Then  $\min\{ y_{2,2},y_{2,3}\}
        <\min\{ y_{4,2},y_{4,3}\}$.
        Since one needs to use all elements in
        $\mathit{KBLS}(4)$ to obtain the $C_4$-face-magic
        labeling $Y$ with properties (1) and (2),
        $Y$ is the unique $C_4$-face-magic labeling on $\mathcal{K}_{4,4}$
        Klein bottle labeling equivalent to $X$
         that satisfies properties (1) and (2).
\end{proof}

The two properties of Proposition \ref{prop:LabelEquiv} 
naturally lead to the concept of a properly symmetrized 
$C_4$-face-magic labeling given in
Definition \ref{Defn:ProperlySymmetrized}.

\begin{definition} \label{Defn:ProperlySymmetrized}
    We say that a $C_4$-face-magic labeling given by\\
   $Y=\{y_{i,j}: 1\le i \le 4 \text{ and } 
   1\le j \le 4\}$ 
   on  $\mathcal{K}_{4,4}$
   is \textit{properly symmetrized}
  if $Y$ satisfies 
   \begin{enumerate}
       \item $y_{1,1}=1$ and
        \item $\min\{ y_{2,2},y_{2,3}\}
        <\min\{ y_{4,2},y_{4,3}\}$.
    \end{enumerate}
\end{definition}

We illustrate the use of Proposition \ref{prop:LabelEquiv}
in Example \ref{ex:SeqTo Symmetrize}.

\begin{example} \label{ex:SeqTo Symmetrize}
    Given a $C_4$-face-magic labeling 
    on $\mathcal{K}_{4,4}$, 
    Figure \ref{fig:properlysymmetrized} illustrates 
    the sequence of transformations in $KBLS(4)$ 
    used in Proposition \ref{prop:LabelEquiv} to 
    properly symmetrize this labeling.
    We will find it useful to represent the $4\times 4$
    Klein bottle grid graph by a $4\times 4$
    checkerboard where each cell represents a vertex, and the
    label in each cell is the label of the corresponding 
    vertex of $\mathcal{K}_{4,4}$.

\begin{figure}[hbt!]
\centering

\begin{tikzpicture}
     \node at (-.5, 5)
    {\begin{tabular}{|c|c|c|c|} \hline
        2  &13  &4   &15  \\  \hline
        7  &\textcolor{blue}{12}   &5   &\textcolor{blue}{$\mathbf{10}$}  \\  \hline
        1  &\textcolor{blue}{14}   &3   &\textcolor{blue}{16}  \\  \hline
        \textcolor{red}{8}  &11  &6   &9  \\  \hline
    \end{tabular}};
    \node[right] at (1.3, 5.5) {$y_{(1,1)}\neq 1$};
    \node[right] at (1.3, 4.8) {$\min\{ y_{2,2},y_{2,3}\}
        \nless\min\{ y_{4,2},y_{4,3}\}$};
    \node[left] at (-2, 4.8) {$Y :$};
\node at (-.5,2.5)
    {\begin{tabular}{|c|c|c|c|} \hline
        1  &14  &3   &16  \\  \hline
        8  &\textcolor{blue}{11}   &6   &\textcolor{blue}{$\mathbf{9}$}  \\  \hline
        2  &\textcolor{blue}{13}   &4   &\textcolor{blue}{15}  \\  \hline
        \textcolor{red}{7}  &12  &5   &10  \\  \hline
    \end{tabular}}; 
    \node[left] at (-2, 2.5) {$Z=N(Y):$}; 
    \node[right] at (1.3, 2.8) {$z_{(1,1)}\neq 1$};
    \node[right] at (1.3, 2.1) {$\min\{ z_{2,2},z_{2,3}\}
        \nless\min\{ z_{4,2},z_{4,3}\}$};
    \node at (-.5,0)
       {\begin{tabular}{|c|c|c|c|} \hline
        7  &12  &5   &10  \\  \hline
        2  & \textcolor{blue}{13}  &4   &\textcolor{blue}{15}  \\  \hline
        8  & \textcolor{blue}{11}   &6   &\textcolor{blue}{\bf{9}}  \\  \hline
        \textcolor{red}{1}  &14  &3   &16  \\  \hline
    \end{tabular}};
 \node[left] at (-2, -.1) {$W=H(Z):$}; 
    \node[right] at (1.3, .4) {$w_{(1,1)}= 1$};
    \node[right] at (1.3, -.2) {$\min\{ w_{2,2},w_{2,3}\}
        \nless\min\{ w_{4,2},w_{4,3}\}$};
 \node at (-.5,-2.5)
       {\begin{tabular}{|c|c|c|c|} \hline
        7  &16  &3   &14  \\  \hline
        2  & \textcolor{blue}{$\mathbf{9}$}  &6   &\textcolor{blue}{11}  \\  \hline
        8  & \textcolor{blue}{15}   &4   &\textcolor{blue}{13}  \\  \hline
        \textcolor{red}{1}  &10  &5   &12  \\  \hline
    \end{tabular}};
 \node[left] at (-2, -2.5) {$X=V(W):$}; 
    \node[right] at (1.3, -2.1) {$x_{(1,1)}= 1$};
    \node[right] at (1.3, -2.7) {$\min\{ x_{2,2},x_{2,3}\}
        <\min\{ x_{4,2},x_{4,3}\}$};       
    \end{tikzpicture}

\caption{The properly symmetrized labeling obtained
by the sequence of symmetries listed in 
Proposition \ref{prop:LabelEquiv}.}
\label{fig:properlysymmetrized}
\end{figure}
\end{example}

\begin{definition}\label{defn:StandardLabel}
    We say that a properly symmetrized
    $C_4$-face-magic labeling\\
     $Y=\{y_{i,j}: 1\le i \le 4 \text{ and } 
   1\le j \le 4\}$ on  $\mathcal{K}_{4,4}$
   is a \textit{standard labeling} if
   \begin{enumerate}
       \item $y_{2,1} > y_{2,4}$,
       \item $y_{3,1} < y_{3,4}$, and
       \item $y_{4,1} > y_{4,4}$.
   \end{enumerate}
\end{definition}

\begin{definition}\label{defn:ElemOperns}
     Let  $X=\{x_{i,j}: 1\le i \le 4 \text{ and } 
   1\le j \le 4\}$ be a 
   properly symmetrized $C_4$-face-magic labeling 
   on  $\mathcal{K}_{4,4}$. 
   Let $k\in\{2,3,4\}$.
   An \textit{elementary vertical cycle operation}
   of $X$ is a transformation $\widehat{V}_k$
   of $X$ to a $C_4$-face-magic labeling
   $Y=\widehat{V}_k(X)$ given by
   \begin{itemize}
       \item $y_{i,j}=x_{i,5-j}$ if $i=k$, and 
       \item $y_{i,j}=x_{i,j}$ otherwise.
   \end{itemize}

   We say that two properly
   symmetrized $C_4$-face-magic labelings $X$ and $Y$
   on $\mathcal{K}_{4,4}$
   are \textit{vertical cycle equivalent} if
   there exists a sequence of of elementary 
  vertical cycle 
   operations that transforms $X$ to $Y$.
   Since each elementary cycle operation $\widehat{V}_k$
   is an involution, vertical cycle equivalence 
   is an equivalence relation.
\end{definition}

In Proposition \ref{prop:extendedgesums}, we analyze 
the structure 
of a $C_4$-face-magic labeling on $\mathcal{K}_{4,4}$.

\begin{proposition}\label{prop:extendedgesums}
Suppose  $X=\{x_{i,j}: 1\le i \le 4 \text{ and } 
   1\le j \le 4\}$ is a $C_4$-face-magic labeling
   on $\mathcal{K}_{4,4}$ with face-magic value $S$.
   Let $a_1= x_{1,1} + x_{1,2}$,
    $a_2= x_{1,2} + x_{1,3}$, and
     $a_3= x_{1,4} + x_{1,1}$.
        Then
\begin{enumerate}
\item[(i)]
\begin{align*}
   a_1 &= x_{2i-1,2j-1} + x_{2i-1,2j} 
    \text{ \ \  and } \\
  S - a_1 &=  x_{2i,2j-1} + x_{2i,2j} 
    \text{ \ \ \ \ \ \ \ \ for } 1\le i \le 2 
    \text{ and } 1 \le j \le 2.
\end{align*}
\item[(ii)]
\begin{align*}
    a_2 & = x_{2i-1, 2} + x_{2i-1, 3}
    \text{ \ \  and } \\
    S - a_2 & = x_{2i, 2} + x_{2i,3}
    \text{ \ \ \ \ \ \ \ \ for } 1\le i \le 2
\end{align*}
\item[(iii)]
\begin{align*}
    a_3 & = x_{2i-1, 4} + x_{2i-1, 1}
    \text{ \ \  and } \\
    S - a_3 & = x_{2i, 4} + x_{2i,1}
    \text{ \ \ \ \ \ \ \ \ for } 1\le i \le 2
\end{align*}
\end{enumerate}
Furthermore, $2a_1 = a_2 + a_3$. In addition, for $1\le i \le 4$,
\begin{equation} \label{eqn:SameEdgeSum}
    x_{i,1} + x_{i,2} = x_{i,3} + x_{i,4}.
\end{equation}
\end{proposition}

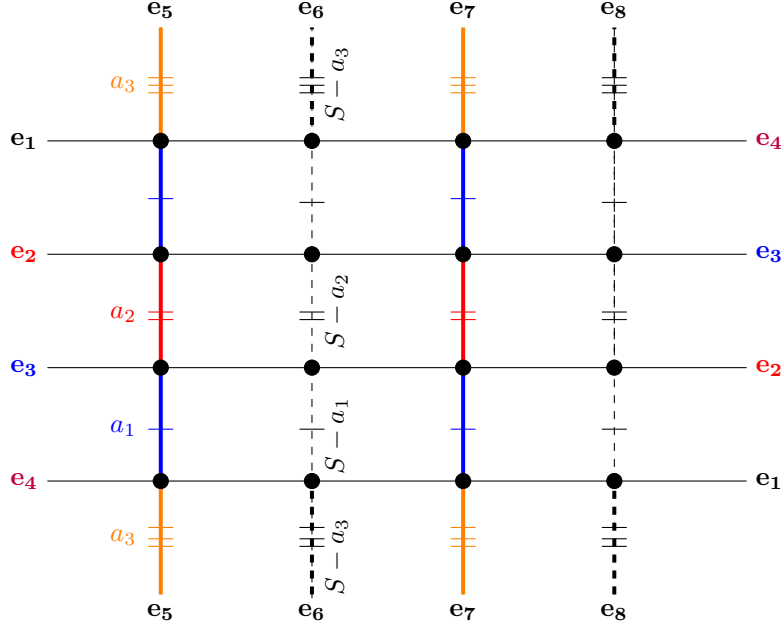
\begin{figure}
\centering
    \begin{tikzpicture}
 \draw (-1.5,0)--(7.75,0);
 \draw (-1.5,1.5)--(7.75, 1.5);
 \draw (-1.5, 3)--(7.75, 3);
 \draw (-1.5, 4.5)--(7.75, 4.5);
 \draw (0,-1.5)--(0, 6);
 \draw[dashed] (2, -1.55)--(2, 1.5);
 \draw[dashed] (2,3)--(2,6);
 \draw[dashed] (4, -1.5)--(4, 6);
 \draw[dashed] (6,-1.5)--(6, 6);
 \draw[dashed] (6,3)--(6, 5.25);
 \draw[dashed] (2,1.5)--(2,3);
 \draw[dashed] (6,1.5)--(6,3);
 \draw[dashed, line width=1.5pt] (2,4.5)--(2,6);
 \draw[dashed, line width=1.5pt] (2,-1.5)--(2,0);
 \draw[dashed, line width=1.5pt] (6,4.5)--(6,6);
 \draw[dashed, line width=1.5pt] (6,-1.5)--(6,0);
  \draw[color=blue, line width=1.5pt] (0,0)--(0,1.5);
 \draw[color=red, line width=1.5pt] (0,1.5)--(0,3);
 \draw[color=blue, line width=1.5pt] (0,3)--(0,4.5);
 \draw[color=orange, line width=1.5pt] (0,-1.5)--(0,0);
 \draw[color=orange, line width=1.5pt] (0,4.5)--(0,6);
 \draw[color=blue, line width=1.5pt] (4,0)--(4,1.5);
 \draw[color=red, line width=1.5pt] (4,1.5)--(4,3);
 \draw[color=blue, line width=1.5pt] (4,3)--(4,4.5);
 \draw[color=orange, line width=1.5pt] (4,-1.5)--(4,0);
 \draw[color=orange, line width=1.5pt] (4,4.5)--(4,6);
 \node at (0,.7) {$\color{blue}{\underline{\hspace{.13in}}}$};
 \node at (0,2.25) {$\color{red}{\underline{\hspace{.13in}}}$};
 \node at (0,2.15) {$\color{red}{\underline{\hspace{.13in}}}$};
 \node at (0,3.75) {$\color{blue}{\underline{\hspace{.13in}}}$};
 \node at (0,5.15) {$\color{orange}{\underline{\hspace{.13in}}}$};
 \node at (0,5.35) {$\color{orange}{\underline{\hspace{.13in}}}$};
 \node at (0,5.25) {$\color{orange}{\underline{\hspace{.13in}}}$};
 \node at (4,.7) {$\color{blue}{\underline{\hspace{.13in}}}$};
 \node at (4,2.25) {$\color{red}{\underline{\hspace{.13in}}}$};
 \node at (4,2.15) {$\color{red}{\underline{\hspace{.13in}}}$};
 \node at (4,3.75) {$\color{blue}{\underline{\hspace{.13in}}}$};
 \node at (4,5.35) {$\color{orange}{\underline{\hspace{.13in}}}$};
 \node at (4,5.15) {$\color{orange}{\underline{\hspace{.13in}}}$};
 \node at (4,5.25) {$\color{orange}{\underline{\hspace{.13in}}}$};
 \node at (0,-.6) {$\color{orange}{\underline{\hspace{.13in}}}$};
 \node at (0,-.75) {$\color{orange}{\underline{\hspace{.13in}}}$};
 \node at (0,-.85) {$\color{orange}{\underline{\hspace{.13in}}}$};
 \node at (4,-.6) {$\color{orange}{\underline{\hspace{.13in}}}$};
 \node at (4,-.75) {$\color{orange}{\underline{\hspace{.13in}}}$};
 \node at (4,-.85) {$\color{orange}{\underline{\hspace{.13in}}}$};
 \node at (2,2.25) {$\color{black}{\underline{\hspace{.13in}}}$};
 \node at (2,2.15) {$\color{black}{\underline{\hspace{.13in}}}$};
 \node at (6,2.25) {$\color{black}{\underline{\hspace{.13in}}}$};
 \node at (6,2.15) {$\color{black}{\underline{\hspace{.13in}}}$};
 \node at (2,5.15) {$\color{black}{\underline{\hspace{.13in}}}$};
 \node at (2,5.35) {$\color{black}{\underline{\hspace{.13in}}}$};
 \node at (2,5.25) {$\color{black}{\underline{\hspace{.13in}}}$};
 \node at (2,-.6) {$\color{black}{\underline{\hspace{.13in}}}$};
 \node at (2,-.75) {$\color{black}{\underline{\hspace{.13in}}}$};
 \node at (2,-.85) {$\color{black}{\underline{\hspace{.13in}}}$};
 \node at (6,-.6) {$\color{black}{\underline{\hspace{.13in}}}$};
 \node at (6,-.75) {$\color{black}{\underline{\hspace{.13in}}}$};
 \node at (6,-.85) {$\color{black}{\underline{\hspace{.13in}}}$};
 \node at (6,5.15) {$\color{black}{\underline{\hspace{.13in}}}$};
 \node at (6,5.35) {$\color{black}{\underline{\hspace{.13in}}}$};
 \node at (6,5.25) {$\color{black}{\underline{\hspace{.13in}}}$};
 \node[above] at (0,6) {$\mathbf{e_5}$};
 \node[below] at (0,-1.5) {$\mathbf{e_5}$};
 \node[above] at (2,6) {$\mathbf{e_6}$};
 \node[below] at (2,-1.5) {$\mathbf{e_6}$};
 \node[above] at (4,6) {$\mathbf{e_7}$};
 \node[below] at (4,-1.5) {$\mathbf{e_7}$};
 \node[above] at (6,6) {$\mathbf{e_8}$};
 \node[below] at (6,-1.5) {$\mathbf{e_8}$};
 \node[left] at (-1.5,4.5) {\textcolor{black}{$\mathbf{e_1}$}};
 \node[right] at (7.75,0) {\textcolor{black}{$\mathbf{e_1}$}};
 \node[left] at (-1.5,3) {\textcolor{red}{$\mathbf{e_2}$}};
 \node[right] at (7.75,1.5) {\textcolor{red}{$\mathbf{e_2}$}};
 \node[left] at (-1.5,1.5) {\textcolor{blue}{$\mathbf{e_3}$}};
 \node[right] at (7.75,3) {\textcolor{blue}{$\mathbf{e_3}$}};
 \node[left] at (-1.5,0) {\textcolor{purple}{$\mathbf{e_4}$}};
 \node[right] at (7.75,4.5) {\textcolor{purple}{$\mathbf{e_4}$}};
 \draw[fill, color=black] (0,0) circle [radius=.1];
 \draw[fill, color=black] (0,1.5) circle [radius=.1];
 \draw[fill, color=black] (0,3) circle [radius=.1];
 \draw[fill, color=black] (0,4.5) circle [radius=.1];
 \draw[fill, color=black] (2,0) circle [radius=.1];
 \draw[fill, color=black] (2,1.5) circle [radius=.1];
 \draw[fill, color=black] (2,3) circle [radius=.1];
 \draw[fill, color=black] (2,4.5) circle [radius=.1];
 \draw[fill, color=black] (4,0) circle [radius=.1];   
 \draw[fill, color=black] (4,1.5) circle [radius=.1];   
 \draw[fill, color=black] (4,3) circle [radius=.1];
 \draw[fill, color=black] (4,4.5) circle [radius=.1];
 \draw[fill, color=black] (6,0) circle [radius=.1];   
 \draw[fill, color=black] (6,1.5) circle [radius=.1];   
 \draw[fill, color=black] (6,3) circle [radius=.1];
 \draw[fill, color=black] (6,4.5) circle [radius=.1];
 \node[left] at (-.2, .7) {$\color{blue}{a_1}$};
 \node[left] at (-.2, 2.2) {$\color{red}{a_2}$};
 \node[left] at (-.2, 5.25) {$\color{orange}{a_3}$};
 \node[left] at (-.2, -.75) {$\color{orange}{a_3}$};
 \node at (2,.7) {$\color{black}{\underline{\hspace{.13in}}}$};
 \node at (6,.7) {$\color{black}{\underline{\hspace{.13in}}}$};
 \node at (2,3.7) {$\color{black}{\underline{\hspace{.13in}}}$};
 \node at (6,3.7) {$\color{black}{\underline{\hspace{.13in}}}$};
\node[rotate=90, anchor=south] at (2.6,2.25) {$S-a_2$};
\node[rotate=90, anchor=south] at (2.6,.6) {$S-a_1$};
\node[rotate=90, anchor=south] at (2.6,5.25) {$S-a_3$};
\node[rotate=90, anchor=south] at (2.6,-1) {$S-a_3$};
\end{tikzpicture}
\caption{The structure of edge sums in 
a $C_4$-face-magic labeling on $\mathcal{K}_{4,4}$ (Proposition \ref{prop:extendedgesums}).}
\label{fig:extendedges}
\end{figure}

\begin{proof}
  As shown in Figure \ref{fig:extendedges} the goal of this proposition is to show that the edge sums 
  along consecutive horizontal edges alternate between $a_i$ and $S-a_i$, for $1\le i\le 3$, in $\mathcal{K}_{4,4}$.
  We will give a proof of (i). The proofs of (ii) and (iii) are similar. 
  The key idea here is that adjacent $C_4$ cycles share a common edge sum which can be subtracted from each in a way that equates $a_1$ with the remaining edge sum (see Figure \ref{fig:extendedges}). Since $\mathcal{K}_{4,4}$ has a $C_4$-face-magic labeling,
  \begin{align}
      x_{1,1}+x_{1,2} +(x_{2,1}+x_{2,2}) & =S= (x_{2,1} + x_{2,2}) + x_{3,1} + x_{3,2}\hspace{.2in} \Rightarrow \label{extnd1}  \\
      a_1=x_{1,1}+x_{1,2}&=x_{3,1}+x_{3,2}. \notag 
  \end{align}
Similarly, by following the edge sums for $e_1$ and $e_2$ from the lower right corner of the representation for $\mathcal{K}_{4,4}$ to the upper right corner we obtain that 
\begin{align}
      x_{3,1}+x_{3,2} +(x_{4,1}+x_{4,2}) & =S= (x_{4,1} + x_{4,2}) + x_{1,3} + x_{1,4}\hspace{.2in} \Rightarrow \label{extnd2}\\
      a_1=x_{3,1}+x_{3,2}&=x_{1,3}+x_{1,4}. \notag
  \end{align}
Finally, 
\begin{align}
      x_{1,3}+x_{1,4} +(x_{2,3}+x_{2,4}) & =S= (x_{2,3} + x_{2,4}) + x_{3,3} + x_{3,4}\hspace{.2in} \Rightarrow   \label{extnd3}\\
      a_1=x_{1,3}+x_{1,4}&=x_{3,3}+x_{3,4}.  \notag
  \end{align}
Solving for the parenthetical expression on the left side of 
the top equations in
(\ref{extnd1}), (\ref{extnd2}), and (\ref{extnd3}) 
together with the bottom equations involving $a_1$, respectively, imply that 
\begin{align*}
    S-a_1&=S-(x_{1,1}+x_{1,2})=x_{2,1} + x_{2,2}\\
         &=S-(x_{3,1}+x_{3,2})=x_{4,1}+x_{4,2}\\
         &=S-(x_{1,3}+x_{1,4})=x_{2,3}+x_{2,4}.
\end{align*}
Also, $x_{4,3}+x_{4,4}=S-(x_{3,3}+x_{3,4})=S-a_1$.
Next, we have
\begin{align*}
    2a_1 &= (x_{1,1} + x_{1,2}) + (x_{1,3} + x_{1,4}) \\
    &= (x_{1,2} + x_{1,3}) + (x_{1,1} + x_{1,4}) = a_2 + a_3.
\end{align*}
Finally, for $1\le i,j\le 2$,
\begin{align*}
    x_{2i-1,2j-1} + x_{2i-1,2j} &= a_1, \text{ and } \\
    x_{2i,2j-1} + x_{2i,2j} &= S - a_1.
\end{align*}
Thus, for $1\le i\le 4$,
\begin{equation*} 
    x_{i,1} + x_{i,2} = x_{i,3} + x_{i,4}.
\end{equation*}
\end{proof}

In Lemma \ref{lem:VerticalCycleOps}, we show that the
an elementary vertical cycle operation transforms a
properly symmetrized $C_4$-face-magic labeling on $\mathcal{K}_{4,4}$ into another
properly symmetrized $C_4$-face-magic labeling on $\mathcal{K}_{4,4}$.

\begin{lemma} \label{lem:VerticalCycleOps}
Let $X=\{ x_{i,j} : 1\le i,j\le 4\}$ on $\mathcal{K}_{4,4}$ 
be a properly symmetrized $C_4$-face-magic labeling on $\mathcal{K}_{4,4}$.
Then, for the elementary vertical cycle operation $\widehat{V}_k$, where $k\in\{2,3,4\}$,
$Y=\widehat{V}_k(X)$ is a properly symmetrized 
   $C_4$-face-magic labeling  on $\mathcal{K}_{4,4}$.
  

     Moreover, no sequence of elementary vertical cycle operations are 
    induced by any of the symmetries in
    $\mathit{KBLS}(4)$ given in Definition \ref{defn:KBSymmetries}.
\end{lemma}

\begin{proof}
We first show that $Y=\widehat{V}_k(X)$ is a $C_4$-face-magic labeling on $\mathcal{K}_{4,4}$.
Since the labels on the vertices $(k,j)$, for $1\le j\le 4$,
are the only ones that are changed, we need to only check
the face sums on faces incident to these vertices.
We have
\begin{equation*}
    y_{k,j} + y_{k,j+1} + y_{k+1,j} + y_{k+1,j+1}=
    x_{k,5-j} + x_{k,5-(j+1)} + x_{k+1,j} + x_{k+1,j+1}/
\end{equation*}
We first observe that
\begin{equation*}
    x_{k,5-j} + x_{k,5-(j+1)} = x_{k,j} + x_{k,j+1}
    \text{ \ for } j\in\{2,4\},
\end{equation*}
where the indices are taken modulo 4.
By \eqref{eqn:SameEdgeSum} of Proposition \ref{prop:extendedgesums},
\begin{equation*}
    x_{k,5-j} + x_{k,5-(j+1)} = x_{k,j} + x_{k,j+1}
    \text{ \ for } j\in\{1,3\}.
\end{equation*}
Hence,
\begin{equation*}
    y_{k,j} + y_{k,j+1} + y_{k+1,j} + y_{k+1,j+1}=
    x_{k,j} + x_{k,j+1} + x_{k+1,j} + x_{k+1,j+1}=S.
\end{equation*}
A similar argument shows that
\begin{equation*}
    y_{k-1,j} + y_{k-1,j+1} + y_{k,j} + y_{k,j+1}=S.
    \end{equation*}

    We observe that $y_{i,j}=x_{i,j}$ for $i\ne k$, and $y_{k,j}=x_{k,5-j}$. Thus,
    \begin{equation*}
        \{y_{i,2}, y_{i,3}\}=\{ x_{i,2}, x_{i,3}\}
 \text{ \ for } 1\le i \le 4.    \end{equation*}
    Hence, 
    \begin{itemize}
        \item $y_{1,1}=x_{1,1}=1$ and
        \item $\min\{y_{2,2},y_{2,3}\} = \min\{x_{2,2},x_{2,3}\}
        < \min\{x_{4,2},x_{4,3}\} = \min\{y_{4,2},y_{4,3}\}$.
    \end{itemize}
    Therefore, $Y$ is a properly symmetrized $C_4$-face-magic
    labeling on $\mathcal{K}_{4,4}$.

    Let $\varphi$ be the composition of a sequence of elementary vertical cycle operations such that 
    $\varphi(X)=\psi(X)$ for some $\psi\in\mathit{KBLS}(4)$.
    We will show that $\psi=I$.
    Since every elementary vertical cycle operation fixes the labels 
    on the vertices $(1,1)$ and $(1,2)$, the graph isomorphism
    $\psi$ fixes the vertices $(1,1)$ and $(1,2)$ on $\mathcal{K}_{4,4}$.
    Let $F_1$ be the face with incident 
    vertices $(1,1)$, $(1,2)$, $(2,1)$ and $(2,2)$, and let $F_2$ be the face
    with incident vertices $(1,1)$, $(1,2)$, $(4,3)$ and $(4,4)$.
    Note that  $\psi$ is the identity map on the edge from $(1,1)$ to $(1,2)$,
    and $\psi$ sends $F_1$ to either $F_1$ or $F_2$.
    If $\psi$ sends $F_1$ to $F_2$, then $\psi=V$ on $\mathcal{K}_{4,4}$.
    So, $\psi=V$ sends the labels on column 2 of $\mathcal{K}_{4,4}$ to those on column 4.
    Since elementary cycle operations send the labels each column to itself,
    $\psi$ sends the labels on column 2 to itself; a contradiction.
    Thus, $\psi$ sends $F_1$ to $F_1$ by the identity map. Hence, $\psi=I$.
\end{proof}

In the next proposition, we observe that for every 
properly symmetrized $C_4$-face-magic labeling on
$\mathcal{K}_{4,4}$, there is a unique standard 
$C_4$-face-magic labeling that is vertical cycle 
equivalent to it.

\begin{proposition} \label{prop:StandardLabel}
    Given a properly symmetrized $C_4$-face-magic labeling $X$ on
    $\mathcal{K}_{4,4}$, there exists a unique
    standard
    $C_4$-face-magic labeling $Y$ on
    $\mathcal{K}_{4,4}$ that is 
    vertical cycle equivalent to $X$.
\end{proposition}

\begin{proof}
 Apply the elementary vertical cycle operations on $X$
 that are needed to obtain a standard labeling.
 Among the eight properly symmetrized labelings that
 are vertical cycle equivalent to $X$,
 only one labeling is a standard labeling.
\end{proof}

We illustrate the use of Proposition 
\ref{prop:StandardLabel} in Example \ref{ex:StandardLabeling}.

\begin{example} \label{ex:StandardLabeling}
    Figure \ref{fig:standardlabeling} illustrates the 
    process of applying elementary vertical cycle 
    operations to a properly symmetrized $C_4$-face magic 
    labeling on $\mathcal{K}_{4,4}$ to obtain
    a standard $C_4$-face magic labeling.

\begin{figure}[hbt!]
\centering

    \begin{tikzpicture}
     \node at (0, 5)
    {\begin{tabular}{|c|c|c|c|} \hline
        7  &\textcolor{red}{$\mathbf{16}$}  &\textcolor{purple}{$\mathbf{5}$}   &\textcolor{brown}{$\mathbf{14}$}  \\  \hline
        2  &\textcolor{blue}{$\mathbf{9}$}   &4   &\textcolor{blue}{$11$}  \\  \hline
        8  &\textcolor{blue}{15}   &6   &\textcolor{blue}{13}  \\  \hline
        \textcolor{orange}{1}  &\textcolor{red}{10}  &\textcolor{purple}{3}   &\textcolor{brown}{12}  \\  \hline
    \end{tabular}};
    \node[right] at (2, 5.8) {$y_{(2,1)}\ngtr y_{(2,4)}$};
    \node[right] at (2, 5.1) {$y_{(3,1)}<y_{(3,4)}$};
    \node[right] at (2, 4.4) {$y_{(4,1)}\ngtr y_{(4,4)}$};
    \node[left] at (-2, 5) {$Y :$}; 
\node at (0,2.5)
{\begin{tabular}{|c|c|c|c|} \hline
        7  &  \textcolor{red}{10} &\textcolor{purple}{$\mathbf{5}$}   &\textcolor{brown}{$\mathbf{14}$}  \\  \hline
        2  & \textcolor{blue}{15}  &4   &\textcolor{blue}{$11$}  \\  \hline
        8  &  \textcolor{blue}{$\mathbf{9}$}  &6   &\textcolor{blue}{13}  \\  \hline
        \textcolor{orange}{1}  &  \textcolor{red}{$\mathbf{16}$}   &\textcolor{purple}{3}   &\textcolor{brown}{12}  \\  \hline
    \end{tabular}};
    \node[left] at (-2, 2.5) {$Z=\widehat{V}_2(Y):$}; 
    \node[right] at (2, 3.2) {$z_{(2,1)}> z_{(2,4)}$};
    \node[right] at (2, 2.5) {$z_{(3,1)}<z_{(3,4)}$};
    \node[right] at (2, 1.8) {$z_{(4,1)}\ngtr z_{(4,4)}$};
    \node[left] at (-2, 5) {$Y :$};
    \node at (0,0)
       {\begin{tabular}{|c|c|c|c|} \hline
        7  &\textcolor{red}{10}  &\textcolor{purple}{$\mathbf{5}$}   &\textcolor{brown}{12}  \\  \hline
        2  &\textcolor{blue}{15}   &4   &\textcolor{blue}{13}  \\  \hline
        8  &\textcolor{blue}{$\mathbf{9}$}   &6   &\textcolor{blue}{11}  \\  \hline
        \textcolor{orange}{1}  &\textcolor{red}{$\mathbf{16}$}  &\textcolor{purple}{3}   &\textcolor{brown}{$\mathbf{14}$}  \\  \hline
    \end{tabular}};
 \node[left] at (-2, .1) {$X=\widehat{V}_4(Z):$}; 
    \node[right] at (2, .5) {$x_{(2,1)}>x_{(2,4)}$};
    \node[right] at (2, -.2) {$x_{(3,1)}<x_{(3,4)}$};
    \node[right] at (2, -.9) {$x_{(4,1)}>x_{(4,4)}$};
    \end{tikzpicture}

\caption{A sequence of vertical cycle operations that transforms 
the properly symmetrized $C_4$-face-magic labeling $Y$ into the
standard labeling $X$.}
\label{fig:standardlabeling}
\end{figure}

\end{example}

\section{Main Results}

In this section we want to characterize the standard 
$C_4$-face-magic labelings on $\mathcal{K}_{4,4}$.

In Proposition \ref{prop:PermissibleA1Values}, we
determine the permissible values of $a_1$ that
are imposed by the structure of edge sums in a $C_4$-face-magic labeling on $\mathcal{K}_{4,4}$ given in
Proposition \ref{prop:extendedgesums}.

\begin{proposition} \label{prop:PermissibleA1Values}
     Suppose $X=\{ x_{i,j} : (i,j) \in V(\mathcal{K}_{4,4})\}$ 
  is a properly symmetrized $C_4$-face-magic labeling on $\mathcal{K}_{4,4}$.
  Then the only permissible values of $a_1$ 
  are $a_1\in\{9,13,15,16,17\}$.
  Furthermore, for each permissible value of $a_1$,
  the corresponding permissible values of $a_2$ are given below.
  \begin{itemize}
    \item If $a_1=9$, then $a_2\in\{10,11,13\}$.
    \item If $a_1=13$, then $a_2\in\{14,15,21\}$.
    \item If $a_1=15$, then $a_2\in\{16,19,23\}$.
    \item If $a_1=16$, then $a_2\in\{18,20,24\}$.
    \item If $a_1=17$, then $a_2\in\{18,19,21,25\}$.
  \end{itemize}
\end{proposition}

\begin{proof}
     By Proposition \ref{prop:extendedgesums},
    the collection
\begin{equation}\label{partition1to16}
    \big\{ \{x_{2i-1,2j-1}, x_{2i-1,2j}\} : 1\le i,j\le 2\big\}
    \cup
     \big\{ \{x_{2i,2j-1}, x_{2i,2j}\} : 1\le i,j\le 2\big\}
\end{equation}
is a partition of the set $\{1,2,\ldots,16\}$ such that
$x_{2i-1,2j-1} + x_{2i-1,2j} = a_1$ and
$x_{2i,2j-1} + x_{2i,2j} =S- a_1$.
Thus, we need to partition the set $\mathcal{C}=\{1,2,\ldots,16\}$ into eight 2-element sets of distinct elements such that
  \begin{enumerate}
    \item[(i)]  the sum of the elements of a set is $a_1$ for four of the 2-element sets, and
    \item[(ii)]  the sum of the elements of a set  is $S-a_1$ for the other four 2-element sets.
  \end{enumerate}
   Note that $a_1\ge 9$ since the smallest positive integer which can be written as a sum of 2 distinct positive integers in exactly four different ways is $9$. Recall that $x_{1,1}=1$. Since $a_1=x_{1,1}+x_{1,2}=1+x_{1,2}$\ and $x_{1,2}\le 16$, then $a_1\le 17$.
  Thus the only possible values for $a_1$ are $9,10,11,12,13,14,15,16$ and $17$.\\
  \hspace{.1in}\\
  Consider the proof that $a_1$ can equal $9$.
  Since $9=1+8=2+7=3+6=4+5$, we will form four distinct two element sets with the integers in each of these four sums: $\{1,8\}, \{2,7\}, \{3, 6\}, \{4,5\}$; which satisfies (i). When these pairs are removed from $\mathcal{C}$ the remaining integers are 
  $9,10,11,12,13,14, 15$ and $16$. There is precisely one way to form four pairs of these integers such that their sum is $34-9=25$, which is $\{9, 16\}, \{10, 15\}, \{11, 14 \}$ and $ \{12, 13\}$; which satisfies (ii). Thus, $a_1$ can equal $9$.\\
  \hspace{2in}\\
  Similarly, as demonstrated in Figure \ref{permissiblevaluesa1} we can partition $\mathcal{C}$ so that (i) and (ii) hold for $a_1=13, 15, 16, 17$.\\
  \hspace{2in}\\

\begin{figure}[hbt!]
    \centering
\renewcommand{\arraystretch}{1.6}
\begin{tabular}{|c|c|c|c|} \hline
$a_1$  &  \text{Pairs in $\mathcal{C}$ that sum to $a_1$} &  $S-a_1$ & \text{Pairs in $\mathcal{C}$ that sum to $S-a_1$}\\  \hline
 $9$  &  $\{1,8\}, \{ 2, 7 \}, \{ 3, 6 \}, \{ 4, 5  \}$  & $25$ & $\{9,16\}, \{10, 15\}, \{  11, 14\}, \{ 12, 13 \}$ \\ \hline
 $13$ & $\{1,12\}, \{ 2, 11 \}, \{ 3, 10\}, \{ 4, 9  \}$ & $21$  & $\{5,16\}, \{ 6, 15 \}, \{ 7, 14\}, \{ 8, 13  \}$\\  \hline
 $15$ &  $\{1,14\}, \{ 2, 13 \}, \{ 5, 10 \}, \{ 6, 9  \}$ &  $19$  &  $\{7,12\}, \{ 8, 11 \}, \{ 3, 16 \}, \{ 4, 15  \}$ \\  \hline
 $16$  &  $\{1,15\}, \{ 3, 13 \}, \{ 5, 11\}, \{ 7, 9  \}$  &  $18$  & $\{2,16\}, \{ 4, 14 \}, \{ 6, 12 \}, \{ 8, 10  \}$  \\   \hline
 $17$ & $\{1,16\}, \{ 2, 15 \}, \{ 3, 14 \}, \{ 4, 13  \}$ & $17$  &  $\{5,12\}, \{ 6, 11 \}, \{ 7, 10\}, \{ 8, 9  \}$ \\ \hline
\end{tabular}
\caption{The four pairs of elements in $\mathcal{C}$ that sum to $a_1$ and $S-a_1$.}
    \label{permissiblevaluesa1}
\end{figure}
 We can eliminate $10, 11, 12$, and $14$ as possible values for $a_1$ as follows.
 Note that $a_1\neq 10$ as follows. Since we would be required to set $S-a_1=24$ if $a_1=10$, there is no integer in $\mathcal{C}$ which differs from $5$ that can be added to $5$ to total $10$ or $24$ since only $5+5=10$ and $5+19=24$.\\
  \hspace{2in}\\
To show that $a_1\neq 11$ we will establish a contradiction by assuming $a_1=11$. The pairs of distinct elements in $\mathcal{C}$ that sum to $a_1=11$ are 
\begin{equation}
    \{1, 10\}, \{2, 9  \}, \{ 3, 8 \}, \{ 4, 7  \}, \text{and}\ \{5, 6\};   \label{pairsumsa1}
\end{equation}
and the pairs of elements in $\mathcal{C}$ that sum to $S-a_1=23$ are  
\begin{equation}
   \{7, 16\}, \{8 ,15\}, \{9, 14\}, \{10, 13\}, \text{and}\ \{11, 12\}. \label{pairsumsS-a1} 
\end{equation} 
By \eqref{partition1to16} we need to choose four pairs from \eqref{pairsumsa1} and four pairs from \eqref{pairsumsS-a1} which when united form a partition of $\mathcal{C}$. Since $\{1, 10\}$ is the only pair with $1$ and $\{10, 13\}$ is the only pair with $13$ we need to include each of these pairs in the partition of $\mathcal{C}$. However, $\{1,10\}\cap \{10,13\}=\{10\}$, which contradicts the requirement that subsets of a partition are disjoint. Therefore $a_1\neq 11$. By similar reasoning, $a_1\neq 12, 14$.\\

We will now determine the permissible values for the edge sum $a_2$ for $a_1=9$. 
Since $x_{1,2}=8$ and $x_{1,3}\ge 2$ then $a_2=x_{1,2}+x_{1,3}\ge 10$. 
Note that since $x_{1,3}+x_{1,4}=9$ and $x_{1,4}\ge 2$ then $x_{1,3}\le 7$. 
Thus, $a_2=x_{1,2}+x_{1,3}\le 15$; so $a_2$ can potentially be $10, 11, 12, 13, 14$, or $15$.\\ 
\hspace{2in}\\
Of the sets $\{1,8\}$
$\{2,7\}$, $\{3,6\}$ and $\{4,5\}$,
let $A_1=\{1,8\}$, choose $A_2$ to be the set   
such that $x_{1,3}\in A_2$, and let $A_3$ and $A_4$ be the remaining two sets.
By Proposition \ref{prop:extendedgesums},
there are labels $y\in A_3$ and $z\in A_4$
such that $y+z=a_2$.
We test this criterion for the 
possible values of the label $x_{1,3}$ from
the set $\{2,3,4,5,6,7\}$.
In Figure  \ref{establish_a2_values}, we observe that
this criterion is only satisfied for $x_{1,3}\in \{2,3,5\}$.
Thus, $a_2\in\{10,11,13\}$. 
In Figure  \ref{eliminate_a2_values}, we see that 
this criterion is not satisfied for $x_{1,3}\in \{4,6,7\}$.
Thus, $a_2\not\in\{12,14,15\}$. 

\begin{figure}[hbt!]
    \centering
\renewcommand{\arraystretch}{1.5}
\begin{tabular}{|c|c|c|c|c|c|} \hline
$A_1$ & $A_2$  & $\textcolor{red}{x_{1,2}}+\textcolor{red}{x_{1,3}}=a_2$ : & $A_3$ &  $A_4$ & $\textcolor{red}{y}+\textcolor{red}{z} = a_2\ :$\\ 
      &            &   $\textcolor{red}{x_{1,2}}\in A_1$\ ,\ $\textcolor{red}{x_{1,3}}\in A_2$                                      &      &    &
$\textcolor{red}{y}\in A_3$\ ,\ $\textcolor{red}{z}\in A_4$ \\  \hline
$\{1, \textcolor{red}{8}\}$ & $\{\textcolor{red}{2}, 7\}$ & $\textcolor{red}{8}+\textcolor{red}{2}=10$ & $\{ 3,\textcolor{red}{6}\}$ &  $\{ \textcolor{red}{4}, 5  \}$ &  $\textcolor{red}{6}+\textcolor{red}{4}=10$\\  \hline
$\{1,\textcolor{red}{8}\}$ & $\{\textcolor{red}{3}, 6\}$ & $\textcolor{red}{8}+\textcolor{red}{3}=11$ & $\{ 2,\textcolor{red}{7}\}$ &  $\{ \textcolor{red}{4}, 5  \}$ &  $\textcolor{red}{7}+\textcolor{red}{4}=11$\\  \hline
$\{1, \textcolor{red}{8}\}$ & $\{4, \textcolor{red}{5}\}$ & $\textcolor{red}{8}+\textcolor{red}{5}=13$ & $\{ 2,\textcolor{red}{7}\}$ &  $\{ 3, \textcolor{red}{6}  \}$ &  $\textcolor{red}{7}+\textcolor{red}{6}=13$\\  \hline
\end{tabular}
    
    \caption{Using a criterion to establish $10, 11$, and $13$ as possible values for $a_2$ when $a_1=9$.}
    \label{establish_a2_values}
\end{figure}

\begin{figure}[hbt!]
\centering
\renewcommand{\arraystretch}{1.5}
\begin{tabular}{|c|c|c|c|c|c|} \hline
$A_1$ & $A_2$  & $\textcolor{red}{x_{1,2}}+\textcolor{red}{x_{1,3}}=a_2$ & $A_3$ &  $A_4$ & $y+z\neq a_2\ :$\\ 
      &            &   $\textcolor{red}{x_{1,2}}\in A_1$\ ,\ $\textcolor{red}{x_{1,3}}\in A_2$                                     &      &    &
$y\in A_3$\ ,\ $z\in A_4$ \\  \hline     
$\{1, \textcolor{red}{8}\}$ & $\{2,\textcolor{red}{7}\}$ & $\textcolor{red}{8}+\textcolor{red}{7}=15$ & $\{ 3,6\}$ &  $\{ 4, 5  \}$ &  $3+4\neq 15$\\ 
          &     &          &            &                 & $3+5\neq 15$\\ 
          &     &          &            &                 & $6+4\neq 15$\\   &     &          &            &                 & $6+5\neq 15$\\  \hline     
 $\{1, \textcolor{red}{8}\}$& $\{3,\textcolor{red}{6}\}$ & $\textcolor{red}{8}+\textcolor{red}{6}=14$ & $\{ 2,7\}$ &  $\{ 4, 5  \}$ &  $2+4\neq 14$\\ 
          &     &          &            &                 & $2+5\neq 14$\\ 
          &     &          &            &                 & $7+4\neq 14$\\   &     &          &            &                 & $7+5=14$\\     \hline
$\{1, \textcolor{red}{8}\}$ & $\{\textcolor{red}{4},5\}$ & $\textcolor{red}{8}+\textcolor{red}{4}=12$ & $\{ 2,7\}$ &  $\{ 3, 6  \}$ &  $2+3\neq 12$\\ 
          &     &          &            &                 & $2+6\neq 12$\\ 
          &     &          &            &                 & $7+3\neq 12$\\   &     &          &            &                 & $7+6\neq 12$\\     \hline          
\end{tabular}
\caption{Using a criterion to eliminate $15, 14$ and $12$ as possible values for $a_2$ when $a_1=9$ by summing all possible values for $y$ and $z$.}
\label{eliminate_a2_values}    
\end{figure}

The proofs for the permissible values of $a_2$ when $a_1$ is $13, 15$, and $16$ are similar to the case where $a_1=9$ (see the tables in Figures \ref{establish_a2_valuesais13}, \ref{establish_a2_valuesais15}, and  \ref{establish_a2_valuesais16}).

\begin{figure}[hbt!]
    \centering
\renewcommand{\arraystretch}{1.5}
\begin{tabular}{|c|c|c|c|c|c|} \hline
$A_1$ & $A_2$  & $\textcolor{red}{x_{1,2}}+\textcolor{red}{x_{1,3}}=a_2$ : & $A_3$ &  $A_4$ & $\textcolor{red}{y}+\textcolor{red}{z} = a_2\ :$\\ 
      &            &   $\textcolor{red}{x_{1,2}}\in A_1$\ ,\ $\textcolor{red}{x_{1,3}}\in A_2$                                      &      &    &
$\textcolor{red}{y}\in A_3$\ ,\ $\textcolor{red}{z}\in A_4$ \\  \hline
$\{1, \textcolor{red}{12}\}$ & $\{\textcolor{red}{2}, 11\}$ & $\textcolor{red}{12}+\textcolor{red}{2}=14$ & $\{ 3,\textcolor{red}{10}\}$ &  $\{ \textcolor{red}{4}, 9  \}$ &  $\textcolor{red}{10}+\textcolor{red}{4}=14$\\  \hline
$\{1,\textcolor{red}{12}\}$ & $\{\textcolor{red}{3}, 10\}$ & $\textcolor{red}{12}+\textcolor{red}{3}=15$ & $\{ 2,\textcolor{red}{11}\}$ &  $\{ \textcolor{red}{4}, 9  \}$ &  $\textcolor{red}{11}+\textcolor{red}{4}=15$\\  \hline
$\{1, \textcolor{red}{12}\}$ & $\{4, \textcolor{red}{9}\}$ & $\textcolor{red}{12}+\textcolor{red}{9}=21$ & $\{ 3,\textcolor{red}{10}\}$ &  $\{ 2, \textcolor{red}{11}  \}$ &  $\textcolor{red}{10}+\textcolor{red}{11}=21$\\  \hline
\end{tabular}
    
    \caption{Using a criterion to establish $10, 11$, and $13$ as possible values for $a_2$ when $a_1=13$.}
    \label{establish_a2_valuesais13}
\end{figure}

\begin{figure}[hbt!]
    \centering
\renewcommand{\arraystretch}{1.5}
\begin{tabular}{|c|c|c|c|c|c|} \hline
$A_1$ & $A_2$  & $\textcolor{red}{x_{1,2}}+\textcolor{red}{x_{1,3}}=a_2$ : & $A_3$ &  $A_4$ & $\textcolor{red}{y}+\textcolor{red}{z} = a_2\ :$\\ 
      &            &   $\textcolor{red}{x_{1,2}}\in A_1$\ ,\ $\textcolor{red}{x_{1,3}}\in A_2$                                      &      &    &
$\textcolor{red}{y}\in A_3$\ ,\ $\textcolor{red}{z}\in A_4$ \\  \hline
$\{1, \textcolor{red}{14}\}$ & $\{\textcolor{red}{2}, 13\}$ & $\textcolor{red}{14}+\textcolor{red}{2}=16$ & $\{ 5,\textcolor{red}{10}\}$ &  $\{ \textcolor{red}{6}, 9  \}$ &  $\textcolor{red}{10}+\textcolor{red}{6}=16$\\  \hline
$\{1,\textcolor{red}{14}\}$ & $\{\textcolor{red}{5}, 10\}$ & $\textcolor{red}{14}+\textcolor{red}{5}=19$ & $\{ \textcolor{red}{6}, 9 \}$ &  $\{2, \textcolor{red}{13}  \}$ &  $\textcolor{red}{6}+\textcolor{red}{13}=19$\\  \hline
$\{1, \textcolor{red}{14}\}$ & $\{6, \textcolor{red}{9}\}$ & $\textcolor{red}{14}+\textcolor{red}{9}=23$ & $\{ 2,\textcolor{red}{13}\}$ &  $\{ 5, \textcolor{red}{10}  \}$ &  $\textcolor{red}{13}+\textcolor{red}{10}=23$\\  \hline
\end{tabular}
    
    \caption{Using a criterion to establish $16, 19$, and $23$ as possible values for $a_2$ when $a_1=15$.}
    \label{establish_a2_valuesais15}
\end{figure}

\begin{figure}[hbt!]
    \centering
\renewcommand{\arraystretch}{1.5}
\begin{tabular}{|c|c|c|c|c|c|} \hline
$A_1$ & $A_2$  & $\textcolor{red}{x_{1,2}}+\textcolor{red}{x_{1,3}}=a_2$ : & $A_3$ &  $A_4$ & $\textcolor{red}{y}+\textcolor{red}{z} = a_2\ :$\\ 
      &            &   $\textcolor{red}{x_{1,2}}\in A_1$\ ,\ $\textcolor{red}{x_{1,3}}\in A_2$                                      &      &    &
$\textcolor{red}{y}\in A_3$\ ,\ $\textcolor{red}{z}\in A_4$ \\  \hline
$\{1, \textcolor{red}{15}\}$ & $\{\textcolor{red}{3}, 13\}$ & $\textcolor{red}{15}+\textcolor{red}{3}=18$ & $\{ 5,\textcolor{red}{11}\}$ &  $\{ \textcolor{red}{7}, 9  \}$ &  $\textcolor{red}{11}+\textcolor{red}{7}=18$\\  \hline
$\{1,\textcolor{red}{15}\}$ & $\{\textcolor{red}{5}, 11\}$ & $\textcolor{red}{15}+\textcolor{red}{5}=20$ & $\{ 3,\textcolor{red}{13}\}$ &  $\{ \textcolor{red}{7}, 9  \}$ &  $\textcolor{red}{13}+\textcolor{red}{7}=20$\\  \hline
$\{1, \textcolor{red}{15}\}$ & $\{7, \textcolor{red}{9}\}$ & $\textcolor{red}{15}+\textcolor{red}{9}=24$ & $\{ 5,\textcolor{red}{11}\}$ &  $\{ 3, \textcolor{red}{13}  \}$ &  $\textcolor{red}{11}+\textcolor{red}{13}=24$\\  \hline
\end{tabular}
    
    \caption{Using a criterion to establish $18, 20$, and $24$ as possible values for $a_2$ when $a_1=16$.}
    \label{establish_a2_valuesais16}
\end{figure}
  
  Suppose $a_1=17$. 
  Then $a_1=\tfrac{1}{2}S$ and $S-a_1=\tfrac{1}{2}S$.
  Thus, $X$ is vertically pairwise balanced.
  So, the sets
   $\{ x_{2i-1,2j-1},\, x_{2i-1,2j}\}$ and $\{ x_{2i,2j-1},\, x_{2i,2j}\}$,
 for $i,j=1,2$, are the eight sets $\{k,17-k\}$, for $1\le k \le 8$,
 which we denote by $A_1, \, A_2,\ldots,\, A_8$.
 By Proposition \ref{prop:extendedgesums}, $a_3=2a_1-a_2=S-a_2$. Then $S-a_3=a_2$.
 Hence,
   \begin{equation}\label{eqn:EdgeSumsAreA2}
      a_2=x_{2i-1,2}+x_{2i-1,3}=x_{2i,1}+x_{2i,4}
      \text{ \ for \ }i=1,\, 2.
  \end{equation}
 Thus, we can combine the sets $A_1, \, A_2,\ldots,\, A_8$ into pair sets
 $A_{2k-1},\, A_{2k}$, for $1\le k \le 4$, such that
there exist $y\in A_{2k-1}$ and $z\in A_{2k}$
  for which $y+z=a_2$.
  Let $A_1=\{ 1,16\}$.
  Then $A_2$ is the set that contains the element 
  $x\in A_2$ for which $x+16=a_2$.

Since $x_{1,2}=16$ and $x_{1,3}\ge 2$, we have
$a_2-x_{1,2}+x_{1,3}\ge 18$.
  By \eqref{eqn:EdgeSumsAreA2}, we have
  \begin{equation*}
      4a_2= \sum_{i=1}^2 (x_{2i-1,2}+x_{2i-1,3} +x_{2i,1}+x_{2i,4})
 \le \sum_{k=9}^{16} k =100.
  \end{equation*}
 Hence, $a_2\le 25$.

 Consider $a_2=18$. 
 Since $x_{1,1}=1$, we have $x_{1,2}=16$.
 Then $x_{1,3}=a_2-x_{1,2}=2$ and $A_2=\{2,\, 15\}$.
 Let $A_3=\{3,\, 14\}$. Since $3$ cannot match with $15$ to add to $a_2=18$,
 $14$ must match with $4$. Thus, $A_4=\{4, 13\}$.
 Let $A_5=\{5,\, 12\}$. Since $5$ cannot match with $13$ to add to $a_2=18$,
 $12$ must match with $6$. Thus, $A_6=\{6, 11\}$.
 Let $A_7=\{7,\, 10\}$. Since $7$ cannot match with $11$ to add to $a_2=18$,
 $10$ must match with $8$. Thus, $A_8=\{8, 6\}$.
 See line 2 of Figure \ref{fig:establish_a2_valuesFora1Is17}.
 A similar analysis establishes the sets $A_1,\, A_2,\ldots,\, A_8$
 for $a_2\in\{19,\, 21, \, 25\}$. See Figure \ref{fig:establish_a2_valuesFora1Is17}.
 
\begin{figure}[ht]
    \centering
    {\small
\renewcommand{\arraystretch}{1.5}
\begin{tabular}{|c|c|c|c|c|} \hline
$A_1, A_2$ & $\textcolor{blue}{x_{1,2}} +\textcolor{red}{x_{1,3}}=a_2$ & $A_3,\, A_4$ &  $A_5,\, A_6$ 
&  $A_7,\, A_8$\\  \hline
$\{1, \textcolor{blue}{16}\}, \{\textcolor{red}{2},15\}$ & $\textcolor{blue}{16}+\textcolor{red}{2}=18$ 
& $\{ 3,\textcolor{blue}{14}\}$, $\{ \textcolor{blue}{4},13\}$  
&  $\{ 5, \textcolor{blue}{12}  \}$, $\{ \textcolor{blue}{6}, 11\}$  
&  $\{7, \textcolor{blue}{10} \}$ , $\{\textcolor{blue}{8}, 9 \}$ \\  \hline
$\{1, \textcolor{blue}{16}\}, \{\textcolor{red}{3},14\}$ & $\textcolor{blue}{16}+\textcolor{red}{3}=19$ 
& $\{ 5,\textcolor{blue}{12}\}$, $\{ \textcolor{blue}{7},10\}$  
&  $\{ \textcolor{blue}{8}, 9  \}$, $\{6, \textcolor{blue}{11} \}$  
&  $\{\textcolor{blue}{4}, 13 \}$ , $\{2, \textcolor{blue}{15}\}$ \\  \hline
$\{1, \textcolor{blue}{16}\},   \{\textcolor{red}{5},12\}$ & $\textcolor{blue}{16}+\textcolor{red}{5}=21$ 
& $\{\textcolor{blue}{8},9\}$, $\{4, \textcolor{blue}{13}\}$  
&  $\{2, \textcolor{blue}{15}  \}$, $\{\textcolor{blue}{6}, 11 \}$  
&  $\{\textcolor{blue}{7},10 \}$ , $\{3, \textcolor{blue}{14}\}$ \\  \hline
$\{1, \textcolor{blue}{16}\},    \{8,\textcolor{red}{9}\}$ & $\textcolor{blue}{16}+\textcolor{red}{9}=25$ 
& $\{ 2,\textcolor{blue}{15}\}$, $\{7, \textcolor{blue}{10}\}$  
&  $\{3, \textcolor{blue}{14}\}$, $\{6, \textcolor{blue}{11} \}$  
&  $\{4, \textcolor{blue}{13} \}$ , $\{5, \textcolor{blue}{12}\}$ \\  \hline
\end{tabular}
    }
    \caption{Using a criterion to establish $18,\, 19, \, 21$, and $25$ as possible values for $a_2$.
    We have $A_1=\{1,\, \textcolor{blue}{16}\}$ where $\textcolor{blue}{x_{1,2}=16}$.
    The label $\textcolor{red}{x_{1,3}}\in A_2$ in highlighted in  \textcolor{red}{red}.
    For the pairs $A_{2k-1}, A_{2k}$, for $k=2,\, 3,\, 4$,
    the labels from each set that sum to $a_2$ are highlighted in 
    \textcolor{blue}{blue}.}
    
    \label{fig:establish_a2_valuesFora1Is17}
\end{figure}

Finally, we need to show $a_2\not\in\{ 20, 22, 23, 24\}$.
Suppose $a_2=20$. Then $x_{1,3}=a_2-x_{1,2}=6$ and $x_{1,4}=a_1-x_{1,3}=11$. Thus $A_1=\{ 1, 16\}$
and $A_2=\{ 4, 13\}$. Let $A_3=\{ 7,10\}$.
Since $10+10=20$, we cannot use use the label $10$ to
use to add to $a_2=20$. Since $7+13=20$, we have 
$A_4=\{ 4,13\}$. Hence, $A_2=\{ 4,13\}=A_4$ which contradicts $A_2\ne A_4$. See Figure \ref{fig:establish_a2_NonvaluesFora1Is17}.
The proofs of the cases $a_2\in\{ 22, 23, 24\}$
are similar. We leave the details to the reader.

\begin{figure}[ht]
    \centering
\renewcommand{\arraystretch}{1.5}
\begin{tabular}{|c|c|c|c|} \hline
$A_1, A_2$ & $\textcolor{blue}{x_{1,2}} +\textcolor{red}{x_{1,3}}=a_2$ & $A_3,\, A_4$ &  $A_5,\, A_6$ \\  \hline
$\{1, \textcolor{blue}{16}\}, \{\textcolor{red}{4},13\}$ & $\textcolor{blue}{16}+\textcolor{red}{4}=20$ 
&  $\{\textcolor{blue}{7}, 10 \}$ , 
$\{4, \textcolor{blue}{13} \}$ & \\  \hline

$\{1, \textcolor{blue}{16}\}, \{\textcolor{red}{6},11\}$ & $\textcolor{blue}{16}+\textcolor{red}{6}=22$ 
&  $\{3, \textcolor{blue}{14}  \}$, $\{\textcolor{blue}{8}, 9 \}$  
&  $\{4, \textcolor{blue}{13} \}$ , $\{8, \textcolor{blue}{9}\}$ \\  \hline

$\{1, \textcolor{blue}{16}\},   \{\textcolor{red}{7},10\}$ & $\textcolor{blue}{16}+\textcolor{red}{7}=23$ 
&  $\{4, \textcolor{blue}{13} \}$ , $\{7, \textcolor{blue}{10}\}$  & \\  \hline

$\{1, \textcolor{blue}{16}\},    \{\textcolor{red}{8}, 9\}$ & $\textcolor{blue}{16}+\textcolor{red}{8}=24$ 
& $\{ 2,\textcolor{blue}{15}\}$, $\{8, \textcolor{blue}{9}\}$ & \\  \hline
\end{tabular}
    
    \caption{Using a criterion to establish $20,\, 22, \, 23$, and $24$ are not possible values for $a_2$.
    We have $A_1=\{1,\, \textcolor{blue}{16}\}$ where $\textcolor{blue}{x_{1,2}=16}$.
    The label $\textcolor{red}{x_{1,3}}\in A_2$ in highlighted in  \textcolor{red}{red}.
    For the pairs $A_{2k-1}, A_{2k}$, for $k=2$ 
    (and $k=3$ if necessary),
    the labels from each set that sum to $a_2$ are highlighted in 
    \textcolor{blue}{blue}.
    Note that for each value of $a_2$, we are forced to duplicate a set $A_j$ which produces a contradiction.}
    \label{fig:establish_a2_NonvaluesFora1Is17}
\end{figure}
\end{proof}

\begin{proposition} 
\label{propHorizontalPairwiseBalancedCondition_A1=9}
  Suppose $X=\{ x_{i,j} : (i,j) \in V(\mathcal{K}_{4,4})\}$ 
  is a standard $C_4$-face-magic labeling on $\mathcal{K}_{4,4}$ 
  such that $a_1=x_{1,1}+x_{1,2}=9$.
  Then $X$ is one of the 
  horizontally pairwise balanced labelings appearing in
  Figure \ref{fig:StandHorzLabelsA1=9}. 
\end{proposition}

\begin{figure}[hbt!]
    \centering
    \begin{tabular}{|c|c|c|c|}\hline
        7  &10  &5   &12  \\  \hline
        2  &15   &4   &13  \\  \hline
        8  &9   &6   &11  \\  \hline
        1  &16  &3   &14  \\  \hline
    \end{tabular}
\hspace{0.15in}\begin{tabular}{|c|c|c|c|} \hline
        6  &11  &5   &12  \\  \hline
        3  &14   &4   &13  \\  \hline
        8  &9   &7   &10  \\  \hline
        1  &16  &2   &15  \\  \hline
    \end{tabular}
 \hspace{0.15in}   \begin{tabular}{|c|c|c|c|} \hline
        4  &13  &3  &14  \\  \hline
        5  &12   &6   &11  \\  \hline
        8  &9   &7   &10  \\  \hline
        1  &16  &2   &15  \\  \hline
    \end{tabular}
    \caption{The standard horizontally pairwise balanced labelings on
    $\mathcal{K}_{4,4}$ with $a_1=9$.}
    \label{fig:StandHorzLabelsA1=9}
\end{figure}

\begin{proof}
By Proposition \ref{prop:extendedgesums},
    the collection
\begin{equation}\label{ppartition1to16}
    \big\{ \{x_{2i-1,2j-1}, x_{2i-1,2j}\} : 1\le i,j\le 2\big\}
    \cup
     \big\{ \{x_{2i,2j-1}, x_{2i,2j}\} : 1\le i,j\le 2\big\}
\end{equation}
is a partition of the set $\{1,2,\ldots,16\}$ such that
$x_{2i-1,2j-1} + x_{2i-1,2j} = a_1$ and
$x_{2i,2j-1} + x_{2i,2j} =S- a_1$.

  There is only one way to partition the set $\{1,2,\ldots,16\}$ into eight 2-element sets such that
  \begin{itemize}
    \item  the sum of the elements of a set is $a_1=9$ for four of the 2-element sets, and
    \item  the sum of the elements of a set  is $S-a_1=25$ for the other four 2-element sets.
  \end{itemize}
  As displayed in the first row of the table in Figure \ref{permissiblevaluesa1}, 
  this unique partition of the set $\{1,2,\ldots,16\}$ is given by the four sets,
  \begin{equation*}
    \{1,8\},  \{2,7\}, \{3,6\}, \mbox{ and } \{4,5\},
  \end{equation*}
  whose sum of elements is $a_1=9$, and by the  four sets,
  \begin{equation*}
    \{9,16\},  \{10,15\}, \{11,14\}, \mbox{ and } \{12,13\},
  \end{equation*}
  whose sum of elements is $S-a_1=25$.

  By Proposition \ref{prop:extendedgesums}, 
  the elements in each 2-element set form the labels on the vertices $(i,2j-1)$ and $(i,2j)$
  for some $1 \le i \le 4$ and $ 1 \le j \le 2$.
  See Figures \ref{figA1EqualsNine} and \ref{figA1EqualsTwentyFive}.
  We adjoin the vertices from two pairs of edges in Figure \ref{figA1EqualsNine} to form two paths on four vertices
  such that the edge sum on each new edge is $a_2$.
  \begin{figure}
\hspace{0.50in}\begin{picture}(300,30)(-50,0)
\put(-50,20){\circle*{7}}
\put(-53,0){$1$}
\put(-50,20){\line(1,0){40}}

\put(-10,20){\circle*{7}}
\put(-13,0){$8$}

\put(30,20){\circle*{7}}
\put(27,0){$2$}
\put(30,20){\line(1,0){40}}

\put(70,20){\circle*{7}}
\put(67,0){$7$}

\put(110,20){\circle*{7}}
\put(107,0){$3$}
\put(110,20){\line(1,0){40}}

\put(150,20){\circle*{7}}
\put(147,0){$6$}

\put(190,20){\circle*{7}}
\put(187,0){$4$}
\put(190,20){\line(1,0){40}}

\put(230,20){\circle*{7}}
\put(227,0){$5$}
\end{picture}
\caption{Edges with edge sum $a_1=9$.}
\label{figA1EqualsNine}
\end{figure}
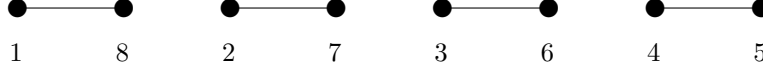
Simultaneously,  we adjoin the vertices from two pairs of edges in Figure \ref{figA1EqualsTwentyFive}
to form two paths on four vertices such that the edge sum on each new edge is $S-a_2$.

\begin{figure}
\hspace{0.50in}\begin{picture}(300,30)(-50,0)
\put(-50,20){\circle*{7}}
\put(-53,0){$9$}
\put(-50,20){\line(1,0){40}}

\put(-10,20){\circle*{7}}
\put(-13,0){$16$}

\put(30,20){\circle*{7}}
\put(27,0){$10$}
\put(30,20){\line(1,0){40}}

\put(70,20){\circle*{7}}
\put(67,0){$15$}

\put(110,20){\circle*{7}}
\put(107,0){$11$}
\put(110,20){\line(1,0){40}}

\put(150,20){\circle*{7}}
\put(147,0){$14$}

\put(190,20){\circle*{7}}
\put(187,0){$12$}
\put(190,20){\line(1,0){40}}

\put(230,20){\circle*{7}}
\put(227,0){$13$}
\end{picture}
\caption{Edges with edge sum $S-a_1=25$.}
\label{figA1EqualsTwentyFive}
\end{figure}
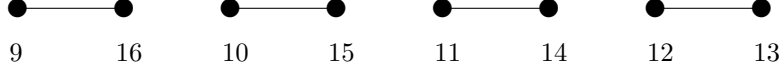

By Proposition \ref{prop:PermissibleA1Values}, the permissible values of $a_2$ for $a_1=9$ are $10, 11$, and $13$. We will construct a labeling for the case $a_2=10$, since constructions using the other values of $a_2$ are similar.
Of the sets $\{9,16\}$
$\{10,15\}$, $\{11,14\}$ and $\{12,13\}$,
let $B_1=\{9,16\}$, choose $B_2$ to be the set   
such that we have $y\in B_1$ and $z\in B_2$ 
such that $y+z=S-a_2=24$, and let $B_3$ and $B_4$ be the remaining two sets.
Since $8\in A_1$ is matched with $2\in A_2$ so that
$x_{1,2}+x_{1,3}=8+2=a_2$ (see Figure \ref{establish_a2_values}), we must match $y=9\in B_1$
with $z\in B_2$ so that $y+z=24$.
Thus, $z=15$ and $B_2=\{10,15\}$.
We need to choose labels $u\in \{11,14\}$
and $w\in\{12,13\}$ so that $u+w=24$.
Thus, $u=11\in B_3=\{11,14\}$ and $w=13\in B_4=\{12,13\}$.
See Figure \ref{fig:LabelA1is9A2is10}.
By Proposition \ref{prop:LabelEquiv}, the path with labels 1, 8, 2, 7 appears in column 1, 
the path with labels 16, 9, 15, 10 appears in column 2,
the path with labels 3, 6, 4, 5 appears in column 3,
and
the path with labels 14, 11, 13, 12 appears in column 4.
We then apply elementary vertical cycle operations to
columns 2, 3 and 4 so that the resulting labeling is a standard labeling.
This labeling is illustrated in Figure  
\ref{fig:LabelA1is9A2is10}.
We observe that the labeling is horizontally pairwise balanced.

\begin{figure}
\centering
\begin{tikzpicture}
 \draw (0,0)--(0, 4.5);
 \draw (2, 0)--(2, 4.5);
 \draw (4, 0)--(4, 4.5);
 \draw (6, 0)--(6, 4.5);
 \draw[fill, color=blue] (0,0) circle [radius=.1];
 \draw[fill, color=blue] (0,1.5) circle [radius=.1];
 \draw[fill, color=blue] (0,3) circle [radius=.1];
 \draw[fill, color=blue] (0,4.5) circle [radius=.1];
 \draw[fill, color=red] (2,0) circle [radius=.1];
 \draw[fill, color=red] (2,1.5) circle [radius=.1];
 \draw[fill, color=red] (2,3) circle [radius=.1];
 \draw[fill, color=red] (2,4.5) circle [radius=.1];
 \draw[fill, color=blue] (4,0) circle [radius=.1];   
 \draw[fill, color=blue] (4,1.5) circle [radius=.1];   
 \draw[fill, color=blue] (4,3) circle [radius=.1];
 \draw[fill, color=blue] (4,4.5) circle [radius=.1];
 \draw[fill, color=red] (6,0) circle [radius=.1];   
 \draw[fill, color=red] (6,1.5) circle [radius=.1];   
 \draw[fill, color=red] (6,3) circle [radius=.1];
 \draw[fill, color=red] (6,4.5) circle [radius=.1];
 \node[above] at (0.7,-.25) {$1$};
 \node[above] at (0.7,1.25) {$8$};
 \node[above] at (0.7,2.75) {$2$};
 \node[above] at (0.7,4.25) {$7$};
 \node[above] at (2.8,-.25) {$16$};
 \node[above] at (2.8,1.25) {$9$};
 \node[above] at (2.8,2.75) {$15$};
 \node[above] at (2.8,4.25) {$10$};
 \node[above] at (4.7,-.25) {$3$};
 \node[above] at (4.7,1.25) {$6$};
 \node[above] at (4.7,2.75) {$4$};
 \node[above] at (4.7,4.25) {$5$};
 \node[above] at (6.8,-.25) {$14$};
 \node[above] at (6.8,1.25) {$11$};
 \node[above] at (6.8,2.75) {$13$};
 \node[above] at (6.8,4.25) {$12$};
\end{tikzpicture}
\caption{Paths that create the labels in the four vertical
cycles of $\mathcal{K}_{4,4}$ when $a_1=9$ and $a_2=10$.}
\label{fig:LabelA1is9A2is10}
\end{figure}
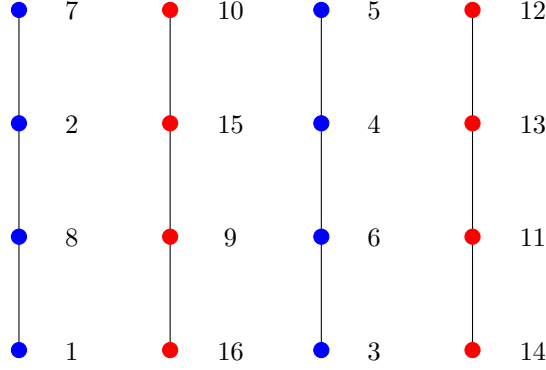
\end{proof}

\begin{proposition} 
\label{propHorizontalPairwiseBalancedCondition_A1=13}
  Suppose $X=\{ x_{i,j} : (i,j) \in V(\mathcal{K}_{4,4})\}$ 
  is a standard $C_4$-face-magic labeling on $\mathcal{K}_{4,4}$ 
  such that $a_1=x_{1,1}+x_{1,2}=13$.
  Then $X$ is one of the 
  horizontally pairwise balanced labelings appearing in
  Figure \ref{fig:StandHorzLabelsA1=13}. 
\end{proposition}

\begin{figure}[hbt!]
    \centering
    \begin{tabular}{|c|c|c|c|} \hline
        11 &6   &9   &8   \\  \hline
        2  &15   &4   &11  \\  \hline
        12 &5   &10  &7   \\  \hline
        1  &16  &3   &14  \\  \hline
    \end{tabular}
\hspace{0.15in}\begin{tabular}{|c|c|c|c|} \hline
        10 &7   &9   &8   \\  \hline
        3  &14   &4   &13  \\  \hline
        12 &5   &11  &6   \\  \hline
        1  &16  &2   &15  \\  \hline
    \end{tabular}
 \hspace{0.15in}   \begin{tabular}{|c|c|c|c|} \hline
        4  &13  &3  &14  \\  \hline
        9  &8    &10  &7  \\  \hline
        12 &5   &11  &6  \\  \hline
        1  &16  &2   &15  \\  \hline
    \end{tabular}
    \caption{The standard horizontally pairwise balanced labelings on
    $\mathcal{K}_{4,4}$ with $a_1=13$.}
    \label{fig:StandHorzLabelsA1=13}
\end{figure}

\begin{proof}
    By Proposition \ref{prop:PermissibleA1Values}, the only
    permissible values of $a_2$ are 14, 15, and 21,
    and the corresponding permissible values of $S-a_2$
    are 20, 19, and 13, respectively.
    An analysis similar to that given in the proof of
    Proposition \ref{propHorizontalPairwiseBalancedCondition_A1=9}
    demonstrates that the labelings in Figure \ref{fig:StandHorzLabelsA1=13} are the only standard
    labelings on $\mathcal{K}_{4,4}$ with $a_1=13$.
\end{proof}

\begin{proposition} 
\label{propHorizontalPairwiseBalancedCondition_A1=15}
  Suppose $X=\{ x_{i,j} : (i,j) \in V(\mathcal{K}_{4,4})\}$ 
  is a standard $C_4$-face-magic labeling on $\mathcal{K}_{4,4}$ 
  such that $a_1=x_{1,1}+x_{1,2}=15$.
  Then $X$ is one of the 
  horizontally pairwise balanced labelings appearing in
  Figure \ref{fig:StandHorzLabelsA1=15}. 
\end{proposition}

\begin{figure}[hbt!]
    \centering
    \begin{tabular}{|c|c|c|c|} \hline
        13 &4   &9   &8   \\  \hline
        2  &15   &6   &11  \\  \hline
        14 &3   &10  &7   \\  \hline
        1  &16  &5   &12  \\  \hline
    \end{tabular}
\hspace{0.15in}\begin{tabular}{|c|c|c|c|} \hline
        10 &7   &9   &8   \\  \hline
        5  &12   &6   &11  \\  \hline
        14 &3   &13  &4   \\  \hline
        1  &16  &2   &15  \\  \hline
    \end{tabular}
 \hspace{0.15in}   \begin{tabular}{|c|c|c|c|} \hline
        6  &11  &5  &12  \\  \hline
        9  &8    &10  &7  \\  \hline
        14 &3   &13  &4  \\  \hline
        1  &16  &2   &15  \\  \hline
    \end{tabular}
    \caption{The standard horizontally pairwise balanced labelings on
    $\mathcal{K}_{4,4}$ with $a_1=15$.}
    \label{fig:StandHorzLabelsA1=15}
\end{figure}

\begin{proof}
    By Proposition \ref{prop:PermissibleA1Values}, the only
    permissible values of $a_2$ are 16, 19, and 23,
    and the corresponding permissible values of $S-a_2$
    are 18, 15, and 11, respectively.
    An analysis similar to that given in the proof of
    Proposition \ref{propHorizontalPairwiseBalancedCondition_A1=9}
    demonstrates that the labelings in Figure \ref{fig:StandHorzLabelsA1=15} are the only standard
    labelings on $\mathcal{K}_{4,4}$ with $a_1=15$.
\end{proof}

\begin{proposition} 
\label{propHorizontalPairwiseBalancedCondition_A1=1}
  Suppose $X=\{ x_{i,j} : (i,j) \in V(\mathcal{K}_{4,4})\}$ 
  is a standard $C_4$-face-magic labeling on $\mathcal{K}_{4,4}$ 
  such that $a_1=x_{1,1}+x_{1,2}=16$.
  Then $X$ is one of the 
  horizontally pairwise balanced labelings appearing in
  Figure \ref{fig:StandHorzLabelsA1=16}. 
\end{proposition}

 \begin{figure}[hbt!]
    \centering
    \begin{tabular}{|c|c|c|c|} \hline
        13 &4   &9   &8   \\  \hline
        3  &14   &7   &10  \\  \hline
        15 &2   &11  &6   \\  \hline
        1  &16  &5   &12  \\  \hline
    \end{tabular}
\hspace{0.15in}\begin{tabular}{|c|c|c|c|} \hline
        11 &6   &9   &8   \\  \hline
        5  &12   &7   &10  \\  \hline
        15 &2   &13  &4   \\  \hline
        1  &16  &3   &15  \\  \hline
    \end{tabular}
 \hspace{0.15in}   \begin{tabular}{|c|c|c|c|} \hline
        7  &10  &5  &12  \\  \hline
        9  &8    &11  &6  \\  \hline
        15 &2   &13  &4  \\  \hline
        1  &16  &3   &15  \\  \hline
    \end{tabular}
    \caption{The standard horizontally pairwise balanced labelings on
    $\mathcal{K}_{4,4}$ with $a_1=16$.}
    \label{fig:StandHorzLabelsA1=16}
\end{figure}

\begin{proof}
    By Proposition \ref{prop:PermissibleA1Values}, the only
    permissible values of $a_2$ are 18, 20, and 24,
    and the corresponding permissible values of $S-a_2$
    are 16, 14, and 10, respectively.
    An analysis similar to that given in the proof of
    Proposition \ref{propHorizontalPairwiseBalancedCondition_A1=9}
    demonstrates that the labelings in Figure \ref{fig:StandHorzLabelsA1=16} are the only standard
    labelings on $\mathcal{K}_{4,4}$ with $a_1=16$.
\end{proof}

\begin{proposition} 
\label{propVerticalPairwiseBalancedCondition_A1=17}
  Suppose $X=\{ x_{i,j} : (i,j) \in V(\mathcal{K}_{4,4})\}$ 
  is a standard $C_4$-face-magic labeling on $\mathcal{K}_{4,4}$ 
  such that $a_1=x_{1,1}+x_{1,2}=17$.
  Then $X$ is one of the twelve (vertically pairwise balanced) labelings displayed 
  in Figures \ref{fig:StandVertLabelsA2=18}, \ref{fig:StandVertLabelsA2=19},
  \ref{fig:StandVertLabelsA2=21}, and \ref{fig:StandVertLabelsA2=25}.
  For convenience, we display the 
  $4 \times 4$ Klein bottle grid graph $\mathcal{K}_{4,4}$
  as a checkerboard where each cell represents a vertex
  and the label in the cell represents the label of the vertex.

  \begin{figure}[hbt!]
    \centering
    \begin{tabular}{|c|c|c|c|} \hline
       15  &6  &13  &8  \\  \hline
        2  &11   &4   &9  \\  \hline
       16  &5   &14  &7  \\  \hline
        1  &12  &3   &10  \\  \hline
    \end{tabular}
\hspace{0.15in}\begin{tabular}{|c|c|c|c|} \hline
       15  &4  &11  &8  \\  \hline
        2  &13   &6   &9  \\  \hline
       16  &3   &12  &7  \\  \hline
        1  &14  &5   &10  \\  \hline
    \end{tabular}
 \hspace{0.15in}   \begin{tabular}{|c|c|c|c|} \hline
       15  &4  &9  &6  \\  \hline
        2  &13   &8   &11  \\  \hline
       16  &3   &10  &5  \\  \hline
        1  &14  &7   &12  \\  \hline
    \end{tabular}
    \caption{The standard vertically pairwise balanced labelings on
    $\mathcal{K}_{4,4}$ with $a_2=18$.}
    \label{fig:StandVertLabelsA2=18}
\end{figure}

\begin{figure}[hbt!]
    \centering
    \begin{tabular}{|c|c|c|c|} \hline
       14  &7  &13  &8  \\  \hline
        3  &10   &4   &9  \\  \hline
       16  &5   &15  &6  \\  \hline
        1  &12  &2   &11  \\  \hline
    \end{tabular}
\hspace{0.15in}\begin{tabular}{|c|c|c|c|} \hline
       14  &4  &10  &8  \\  \hline
        3  &13   &7   &9  \\  \hline
       16  &2   &12  &6  \\  \hline
        1  &15  &5   &11  \\  \hline
    \end{tabular}
\hspace{0.15in}\begin{tabular}{|c|c|c|c|} \hline
       14  &4  &9  &7  \\  \hline
        3  &13   &8   &10  \\  \hline
       16  &2   &11  &5  \\  \hline
        1  &15  &6   &12  \\  \hline
    \end{tabular}
    \caption{The standard vertically pairwise balanced labelings on
    $\mathcal{K}_{4,4}$ with $a_2=19$.}
    \label{fig:StandVertLabelsA2=19}
\end{figure}

\begin{figure}[hbt!]
    \centering
    \begin{tabular}{|c|c|c|c|} \hline
       12  &7  &11  &8  \\  \hline
        5  &10   &6   &9  \\  \hline
       16  &3   &15  &4  \\  \hline
        1  &14  &2   &13  \\  \hline
    \end{tabular}
\hspace{0.15in}\begin{tabular}{|c|c|c|c|} \hline
       12  &6 &10  &8  \\  \hline
        5  &11   &7   &9  \\  \hline
       16  &2   &14  &4  \\  \hline
        1  &15  &3   &13  \\  \hline
    \end{tabular}
\hspace{0.15in}\begin{tabular}{|c|c|c|c|} \hline
       12  &6  &9  &7  \\  \hline
        5  &11   &8   &10  \\  \hline
       16  &2   &13  &3  \\  \hline
        1  &15  &4   &14  \\  \hline
    \end{tabular}
    \caption{The standard vertically pairwise balanced labelings on
    $\mathcal{K}_{4,4}$ with $a_2=21$.}
    \label{fig:StandVertLabelsA2=21}
\end{figure}

\begin{figure}[hbt!]
    \centering
    \begin{tabular}{|c|c|c|c|} \hline
       8  &11  &7  &12  \\  \hline
        9  &6   &10   &5  \\  \hline
       16  &3   &15  &4  \\  \hline
        1  &14  &2   &13  \\  \hline
    \end{tabular}
\hspace{0.15in} \begin{tabular}{|c|c|c|c|} \hline
       8  &10 &6  &12  \\  \hline
        9  &7   &11   &5  \\  \hline
       16  &2   &14  &4  \\  \hline
        1  &15  &3   &13  \\  \hline
    \end{tabular}
\hspace{0.15in}  \begin{tabular}{|c|c|c|c|} \hline
       8  &10  &5  &11  \\  \hline
        9  &7   &12   &6  \\  \hline
       16  &2   &13  &3  \\  \hline
        1  &15  &4   &14  \\  \hline
    \end{tabular}
    \caption{The standard vertically pairwise balanced labelings on
    $\mathcal{K}_{4,4}$ with $a_2=25$.}
    \label{fig:StandVertLabelsA2=25}
\end{figure}
\end{proposition}

\begin{proof}
By Proposition \ref{prop:PermissibleA1Values},
the only permissible values of $a_2$ are $a_2\in\{18,19,21,25\}$.
Also, for each permissible value of $a_2$, there is only one way
to choose a label from each set $\{i,17-i\}$, for $1\le i\le 8$,
so that there are labels from a pair of these sets that add to $a_2$;
these labels are highlighted in blue 
and red in Figure \ref{fig:establish_a2_valuesFora1Is17}.
We illustrate the construction of the three standard labelings in Figure \ref{fig:StandVertLabelsA2=18} when $a_2=18$.
The construction of the standard labelings in Figures \ref{fig:StandVertLabelsA2=19}, 
\ref{fig:StandVertLabelsA2=21} and \ref{fig:StandVertLabelsA2=25} are similar;
we leave the details of the proof of the cases $a_2\in\{19,\,21,\,25\}$ to the reader.
Figure \ref{fig:establish_a2_valuesFora1Is17} highlights the desired label
from $A_2$ in red and the labels from the other sets are highlighted in blue. 
For convenience, we duplicate row 2 of Figure \ref{fig:establish_a2_valuesFora1Is17} below.\\
\hspace{2in}\\
\renewcommand{\arraystretch}{1.5}
\begin{tabular}{|c|c|c|c|c|c|} \hline
$A_1$ & $A_2$ & $\textcolor{blue}{x_{1,2}} +\textcolor{red}{x_{1,3}}=a_2$ & $A_3,\, A_4$ &  $A_5,\, A_6$ 
&  $A_7,\, A_8$\\  \hline
$\{1, \textcolor{blue}{16}\}$ &  $\{\textcolor{red}{2},15\}$ & $\textcolor{blue}{16}+\textcolor{red}{2}=18$ 
& $\{ 3,\textcolor{blue}{14}\}$, $\{ \textcolor{blue}{4},13\}$  
&  $\{ 5, \textcolor{blue}{12}  \}$, $\{ \textcolor{blue}{6}, 11\}$  
&  $\{7, \textcolor{blue}{10} \}$ , $\{\textcolor{blue}{8}, 9 \}$ \\  \hline
\end{tabular}\\
\hspace{2in}\\

Since $x_{1,1}=1$ in a standard labeling, column 1 for all 3 labelings is given below. 
\begin{center}
\begin{tabular}{|c|c|c|c|} \hline
       15  &  &  &  \\  \hline
        2  &  &   &  \\  \hline
       16  &   &  & \\  \hline
        1  &  &   &  \\  \hline
    \end{tabular}
    \end{center}
\hspace{2in}\\
There are 3 ways to fill column 3. We can choose the pair sets $A_3, A_4$ or $A_5, A_6$ or $A_7, A_8$. 
The choice of the pair sets $A_3, A_4$ or $A_5, A_6$ or $A_7, A_8$
that we place into column 3, while maintaining the standard condition that $y_{3,1}<y_{3,4}$,
determines which of the three standard labelings we obtain in Figure \ref{fig:StandVertLabelsA2=18}.
We show that the choice of the set pair $A_3, A_4$ in column 3 gives rise to the first
standard labeling in Figure \ref{fig:StandVertLabelsA2=18}.
We leave the details of the proof of the other two standard
labelings in Figure \ref{fig:StandVertLabelsA2=18} that arise from the choices of $A_5, A_6$ or $A_7, A_8$ to the reader.
Then columns 1 and 3 of the standard labeling arising from the choice of $A_3, A_4$ are given below.\\
\hspace{2in}\\
\begin{center}
\begin{tabular}{|c|c|c|c|} \hline
       15  &  & 13 &  \\  \hline
        2  &  & 4  &   \\  \hline
       16  &   & 14 & \\  \hline
        1  &  & 3  &  \\  \hline
    \end{tabular}
\end{center}
    We will fill the remaining columns so that the properly symmetrized conditions hold, namely:
$\min\big\{y_{2,2}, y_{2,3} \big\}<\min\big\{y_{4,2}, y_{4,3} \big\}$. Thus we must place $A_5, A_6$ in column 2 and $A_7, A_8$ in column 4.\\
\hspace{2in}\\
Also, we will place these labels so that the following standard conditions hold: $y_{2,1}>y_{2,4}$ and $y_{4,1}>y_{4,4}$.
\hspace{2in}\\
\begin{center}
\begin{tabular}{|c|c|c|c|} \hline
       15  & 6 & 13 & 8 \\  \hline
        2  & 11 & 4  & 9 \\  \hline
       16  & 5 & 14 & 7\\  \hline
        1  & 12 & 3  &  10\\  \hline
    \end{tabular}
    \end{center}

\end{proof}

\begin{theorem}\label{binaryhd_ks}
 Each of the twelve standard horizontally pairwise balanced 
$C_4$-face-magic labelings on $\mathcal{K}_{4,4}$ can be expressed in terms of
$d_1,\, d_2,\, d_3,$ and $d_4$ as displayed in Figure \ref{fig:SchemeOne} where the values of $d_k$\ are:
\begin{align*}
    d_1 &\in \{ 1,\, 2,\, 4,\, 8\} &&
    d_2 \in \{ 1,\, 2,\, 4,\, 8\}\setminus \{d_1\} \\
    d_3 &= \min\big(\{ 1,\, 2,\, 4,\, 8\}\setminus \{d_1,\, d_2\}\big)
    &&d_4 = \max\big(\{ 1,\, 2,\, 4,\, 8\}\setminus \{d_1,\, d_2\}\big)
\end{align*}
\end{theorem}

\begin{figure}[hbt!]
\centering

\renewcommand{\arraystretch}{2}
\begin{tabular}{|c|c|c|c|}\hline
$\textcolor{red}{1+d_3+d_4}$ & $\textcolor{blue}{1+d_1+d_2}$ & $\textcolor{red}{1+d_4}$ & $\textcolor{blue}{1+d_1+d_2+d_3}$ \\ \hline
$\textcolor{blue}{1+d_2}$ & $\textcolor{red}{1+d_1+d_3+d_4}$ & $\textcolor{blue}{1+d_2+d_3}$ & $\textcolor{red}{1+d_1+d_4}$\\  \hline
$\textcolor{red}{1+d_2+d_3+d_4}$ & $\textcolor{blue}{1+d_1}$ & $\textcolor{red}{1+d_2+d_4}$ & $\textcolor{blue}{1+d_1+d_3}$\\ \hline
$\textcolor{blue}{1}$ &  $\textcolor{red}{16}$ &  $\textcolor{blue}{1+d_3}$  & $\textcolor{red}{1+d_1+d_2+d_4}$\\ \hline
\end{tabular}

\caption{The twelve standard horizontally pairwise balanced 
$C_4$-face-magic labelings on $\mathcal{K}_{4,4}$. 
}\label{fig:SchemeOne}
\end{figure}

\begin{proof}
In this proof we will define $d_1, d_2, d_3,$ and $d_4$ and verify that the labelings  $\{x_{i,j} : \ 1\le i\le 4, 1\le j\le 4\}$ have the expressions indicated in Figure \ref{fig:SchemeOne}. By Proposition \ref{prop:extendedgesums} we know that $x_{1,2}+x_{1,1}=a_1$. Also, since the labeling of $\mathcal{K}_{4,4}$ is horizontally balanced, $x_{2,2}+x_{1,2}=\frac{1}{2}S=17$; hence $x_{2,2}=18-a_1$ by also using the fact that $x_{1,1}=1$. As shown in Figure \ref{fig:SchemeOne} we want $x_{2,2}=d_1+1=d_1+x_{1,1}$. Thus we will define
\begin{equation}
  d_1=x_{2,2}-x_{1,1}=(18-a_1)-1=17-a_1.  \label{defofd1}
\end{equation}
The table in Figure \ref{confirmd_1} confirms that $x_{2,2}=d_1+1$ as displayed in the labelings in Figures \ref{fig:StandHorzLabelsA1=9}, \ref{fig:StandHorzLabelsA1=13}, \ref{fig:StandHorzLabelsA1=15}, and \ref{fig:StandHorzLabelsA1=16}. 

\begin{figure}[hbt!]
\centering
\renewcommand{\arraystretch}{1.5}
\begin{tabular}{|c|c|c|} \hline
       $a_1$  & $d_1$ & $x_{2,2}$ \\  \hline
        9  & 8 & 9  \\  \hline
       13  &  4 & 5 \\  \hline
        15  & 2 & 3  \\  \hline
        16 & 1 & 2 \\ \hline
    \end{tabular}
\caption{Confirming that $x_{2,2}=d_1+1$.}\label{confirmd_1}
\end{figure}
By Proposition \ref{prop:extendedgesums}, $x_{1,2}+x_{1,3}=a_2$; so, since $x_{1,2}=a_1-1$, then $x_{1,3}=(a_2-a_1)+1$. We will therefore define 
\begin{equation}
d_2=a_2-a_1.   \label{defofd2}
\end{equation}
We confirm that $x_{1,3}=d_2+1$ in Figure \ref{confirmd_2} for each of the standard labelings given in Figures 
\ref{fig:StandHorzLabelsA1=9}, \ref{fig:StandHorzLabelsA1=13}, \ref{fig:StandHorzLabelsA1=15}, and \ref{fig:StandHorzLabelsA1=16}.

\begin{figure}[hbt!]
\centering
\renewcommand{\arraystretch}{1.5}
\begin{tabular}{|c|c|c|c|c|} \hline
  $a_1$ & $a_2$  & $d_1$ & $d_2$ & $x_{1,3}$ \\  \hline
    9  &  10 & 8 & 1 & 2  \\  \hline
     9 &  11 & 8  &  2 & 3 \\  \hline
     9 &    13 & 8 & 4 & 5  \\  \hline
    \end{tabular}\hspace{0.15in}
  \begin{tabular}{|c|c|c|c|c|} \hline
  $a_1$ & $a_2$ & $d_1$ & $d_2$ & $x_{1,3}$ \\  \hline
    13  &  14 & 4 & 1 & 2  \\  \hline
     13 &  15 & 4  &  2 & 3 \\  \hline
     13 &    21 & 4  & 8 & 9  \\  \hline
    \end{tabular}  
\hspace{2in}\\
\hspace{2in}\\
\hspace{2in}\\
\begin{tabular}{|c|c|c|c|c|} \hline
  $a_1$ & $a_2$ & $d_1$ & $d_2$ & $x_{1,3}$ \\  \hline
    15  &  16 & 2 & 1 & 2  \\  \hline
     15 &  19 & 2  &  4 & 5 \\  \hline
     15 &    23 & 2  & 8 & 9  \\  \hline
    \end{tabular}\hspace{0.15in}
\begin{tabular}{|c|c|c|c|c|} \hline
  $a_1$ & $a_2$ & $d_1$  & $d_2$ & $x_{1,3}$ \\  \hline
    16  &  18 & 1  & 2 & 3  \\  \hline
     16 &  20 & 1  &  4 & 5 \\  \hline
     16 &    24 & 1  & 8 & 9  \\  \hline
    \end{tabular}
    \caption{Confirming that $x_{1,3}=d_2+1$.}\label{confirmd_2}
\end{figure}

The tables in Figures \ref{confirmd_1} and \ref{confirmd_2} show that $d_2 \in \{ 1,\, 2,\, 4,\, 8\}\setminus \{d_1\}$. We define
\begin{equation*}
  d_3= \min\big(\{ 1,\, 2,\, 4,\, 8\}\setminus \{d_1,\, d_2\}\big)\hspace{.1in}\text{and}\hspace{.1in}d_4= \max\big(\{ 1,\, 2,\, 4,\, 8\}\setminus \{d_1,\, d_2\}\big).   
\end{equation*}
The table in Figure \ref{confirmd_3} verifies that $x_{3,1}=d_3+1$ and $x_{3,4}=d_4+1$ for each of the
standard labelings given in Figures \ref{fig:StandHorzLabelsA1=9}, \ref{fig:StandHorzLabelsA1=13}, 
\ref{fig:StandHorzLabelsA1=15}, and \ref{fig:StandHorzLabelsA1=16}.

\begin{figure}[hbt!]
\centering
\renewcommand{\arraystretch}{1.5}
\begin{tabular}{|c|c|c|c|c|c|c|c|} \hline
  $a_1$ & $a_2$  & $d_1$ & $d_2$ & $\textcolor{blue}{d_3}$ & $\textcolor{red}{d_4}$ & $x_{3,1}$  & $x_{3,4}$\\  \hline
    9  &  10  & 8 & 1 & \textcolor{blue}{2} & \textcolor{red}{4} & 3  & 5 \\  \hline
     9 &  11  &  8 & 2 & \textcolor{blue}{1} & \textcolor{red}{4} & 2 & 5 \\  \hline
     9 &    13  & 8 & 4 & \textcolor{blue}{1} & \textcolor{red}{2} & 2 & 3  \\  \hline
     13  &  14  & 4 & 1 & \textcolor{blue}{2} & \textcolor{red}{8} & 3  & 9 \\  \hline
     13 &  15  &  4 & 2 & \textcolor{blue}{1} & \textcolor{red}{8} & 2 & 9 \\  \hline
     13 &    21  & 4 & 8 & \textcolor{blue}{1} & \textcolor{red}{2} & 2 & 3  \\  \hline
    \end{tabular}\hspace{0.15in}
\begin{tabular}{|c|c|c|c|c|c|c|c|} \hline
  $a_1$ & $a_2$  & $d_1$ & $d_2$ & $\textcolor{blue}{d_3}$ & $\textcolor{red}{d_4}$ & $x_{3,1}$  & $x_{3,4}$\\  \hline
    15  &  16  & 2 & 1 & \textcolor{blue}{4} & \textcolor{red}{8} & 5  & 9 \\  \hline
     15 &  19  &  2 & 4 & \textcolor{blue}{1} & \textcolor{red}{8} & 2 & 9 \\  \hline
     15 &   23  & 2 & 8 & \textcolor{blue}{1} & \textcolor{red}{4} & 2 & 5  \\  \hline
     16  &  18  & 1 & 2 & \textcolor{blue}{4} & \textcolor{red}{8} & 5  & 9 \\  \hline
     16 &  20  &  1 & 4 & \textcolor{blue}{2} & \textcolor{red}{8} & 3 & 9 \\  \hline
     16 &    24  & 1 & 8 & \textcolor{blue}{2} & \textcolor{red}{4} & 3 & 5  \\  \hline
    \end{tabular}

\caption{Confirming that $x_{3,1}=d_3+1$ and $x_4=d_4+1$.}\label{confirmd_3}
\end{figure}
\hspace{2in}\\
We will now use the expressions for $x_{2,2}, x_{1,3}$ and $x_{3,1}$, the known values of $x_{1,1}=1$ and $x_{2,1}=16$, and horizontally pairwise balanced property to derive expressions for the remaining labels for these 12 graphs. Note that $\textcolor{red}{1+d_1+d_2+d_3+d_4}=\textcolor{red}{16}$. Since $x_{3,1}+x_{4,1}=17$, it follows that 
\begin{equation}\label{horizbal1}
  x_{4,1}=17-x_{3,1}=17-( 1+d_3)= \textcolor{red}{16}- d_3 =1+d_1+d_2+d_4.  
\end{equation}
By using computations similar to (\ref{horizbal1}), 
$x_{1,2}=17-x_{2,2}=1+d_2+d_3+d_4$,\\  $x_{2,3}=17-x_{1,3}=1+d_1+d_3+d_4$, and
$x_{4,4}=17-x_{3,4}=1+d_1+d_2+d_3$. By (\ref{defofd1}) $\textcolor{blue}{a_1}=\textcolor{blue}{17-d_1}$. Since $x_{3,1}+x_{3,2}=\textcolor{blue}{a_1}$, we have that
\begin{equation*}
    x_{3,2}=\textcolor{blue}{a_1}-x_{3,1}=(\textcolor{blue}{17-d_1})-(1+d_3)=\textcolor{red}{16}-d_1-d_3=1+d_2+d_4.
\end{equation*}
Similarly, $x_{1,4}=\textcolor{blue}{a_1}-x_{1,3}=1+d_3+d_4$. By Proposition \ref{prop:extendedgesums} we have that
\begin{align*}
    x_{4,2}&=(S-\textcolor{blue}{a_1})-x_{4,1}\\
           &=34-(\textcolor{blue}{17-d_1})-(1+d_1+d_2+d_4)=\textcolor{red}{16}-d_2-d_4=1+d_1+d_3.
\end{align*}
By  (\ref{defofd2}), which implies that $a_2=d_2+a_1$, we obtain
\begin{align*}
     x_{3,3}&=a_2-x_{3,2}=(d_2+\textcolor{blue}{a_1})-x_{3,2}
        =\big(d_2+(\textcolor{blue}{17-d_1})\big)-(1+d_2+d_4)\\
        &=\textcolor{red}{16}-d_1-d_4= 1+d_2+d_3
\end{align*}
Finally, by computations similar to (\ref{horizbal1}) we obtain that
$x_{4,3}=17-x_{3,3}=1+d_1+d_4$ and $x_{2,4}=17-x_{1,4}=1+d_1+d_2$.
\end{proof}

\begin{definition}
In order to formulate an analagous result to Theorem \ref{binaryhd_ks} for the 12 standard vertically pairwise balanced $C_4$-face-magic labelings 
on $\mathcal{K}_{4,4}$ we will need to modify the third labeling in each of Figures \ref{fig:StandVertLabelsA2=18}, \ref{fig:StandVertLabelsA2=19}, 
\ref{fig:StandVertLabelsA2=21}, and \ref{fig:StandVertLabelsA2=25} by applying $\widehat{V}_3$ and $\widehat{V}_4$ to each of those labeling. We will refer 
to the resulting labeling as a \it{pseudo-standard}\rm\ labeling. The four pseudo-standard labelings are listed in Figure \ref{pseudostandardlabeling12} for
reference. We will also refer to the first two labelings in 
Figures \ref{fig:StandVertLabelsA2=18}, \ref{fig:StandVertLabelsA2=19}, 
\ref{fig:StandVertLabelsA2=21}, and \ref{fig:StandVertLabelsA2=25}
as \it{fundamentally standard}\rm\ labelings.  
\end{definition}

\begin{figure}[hbt!]
\centering
\begin{tabular}{|c|c|c|c|} \hline
       15  &4  &\textcolor{red}{7}  &\textcolor{red}{12}  \\  \hline
        2  &13   &\textcolor{red}{1}   &\textcolor{red}{5}  \\  \hline
       16  &3   &\textcolor{red}{8}  &\textcolor{red}{11}  \\  \hline
        1  &14  &\textcolor{red}{9}   &\textcolor{red}{6}  \\  \hline
    \end{tabular}
\hspace{0.15in}\begin{tabular}{|c|c|c|c|} \hline
       14  &4  &\textcolor{red}{6}  &\textcolor{red}{12}  \\  \hline
        3  &13   &\textcolor{red}{11}   &\textcolor{red}{5}  \\  \hline
       16  &2   &\textcolor{red}{8}  &\textcolor{red}{10}  \\  \hline
        1  &15  &\textcolor{red}{9} &  \textcolor{red}{7}  \\  \hline
    \end{tabular}
\hspace{0.15in}\begin{tabular}{|c|c|c|c|} \hline
       12  &6  &\textcolor{red}{4}  &\textcolor{red}{14}  \\  \hline
        5  &11   &\textcolor{red}{13}   &\textcolor{red}{3}  \\  \hline
       16  &2   &\textcolor{red}{8}  &\textcolor{red}{10}  \\  \hline
        1  &15  &\textcolor{red}{9}   &\textcolor{red}{7}  \\  \hline
    \end{tabular}
\hspace{2in}\\
\hspace{2in}\\
\hspace{2in}\\
\hspace{0.15in}  \begin{tabular}{|c|c|c|c|} \hline
       8  &10  &\textcolor{red}{4}  &\textcolor{red}{14}  \\  \hline
        9  &7   &\textcolor{red}{13}   &\textcolor{red}{3}  \\  \hline
       16  &2   &\textcolor{red}{12}  &\textcolor{red}{6}  \\  \hline
        1  &15  &\textcolor{red}{5}   &\textcolor{red}{11}  \\  \hline
    \end{tabular}
\caption{The four pseudo-standard  vertically pairwise 
balanced $C_4$-face-magic labelings 
for $a_2=18, 19, 21, 25$, 
respectively.}\label{pseudostandardlabeling12}   
\end{figure}

\begin{theorem}\label{dksvertical}
 Each of the 8 fundamentally standard and 4 pseudo-standard vertically pairwise balanced 
$C_4$-face-magic labelings on $\mathcal{K}_{4,4}$ can be expressed in terms of
$d_1,\, d_2,\, d_3,$ and $d_4$ as displayed in Figure \ref{fig:SchemeTwo} where the values of $d_k$\ are:
\begin{align*}
    d_1 &\in \{ 1,\, 2,\, 4,\, 8\}  && d_2 \in \{ 1,\, 2,\, 4,\, 8\}\setminus \{d_1\} \\
    d_3 &= \min\big(\{ 1,\, 2,\, 4,\, 8\}\setminus \{d_1,\, d_2\}\big)
    && d_4 = \max\big(\{ 1,\, 2,\, 4,\, 8\}\setminus \{d_1,\, d_2\}\big)
\end{align*}
\end{theorem}

\begin{figure}[hbt!]
\centering

\renewcommand{\arraystretch}{2}
\begin{tabular}{|c|c|c|c|}\hline
$\textcolor{red}{1+d_2+d_3+d_4}$ & $\textcolor{blue}{1+d_1+d_3}$ & $\textcolor{red}{1+d_3+d_4}$ & $\textcolor{blue}{1+d_1+d_2+d_3}$ \\ \hline
$\textcolor{blue}{1+d_1}$ & $\textcolor{red}{1+d_2 +d_4}$ & $\textcolor{blue}{1+d_1+d_2}$ & $\textcolor{red}{1+d_4}$\\  \hline
$\textcolor{red}{16}$ & $\textcolor{blue}{1+d_3}$ & $\textcolor{red}{1+d_1+d_3+d_4}$ & $\textcolor{blue}{1+d_2+d_3}$\\ \hline
$\textcolor{blue}{1}$ &  $\textcolor{red}{1+d_1+d_2+d_4}$ &  $\textcolor{blue}{1+d_2}$  & $\textcolor{red}{1+d_1+d_4}$\\ \hline
\end{tabular}

\caption{Twelve vertically pairwise balanced 
vertical-cycle-nonequivalent
$C_4$-face-magic labelings on $\mathcal{K}_{4,4}$.
} 
\label{fig:SchemeTwo}
\end{figure}

\begin{proof}
  As in the proof of Theorem \ref{binaryhd_ks}, we will define $d_1, d_2, d_3$\ and $d_4$,
  and then verify that the labelings in this theorem have the expressions indicated in Figure \ref{fig:SchemeTwo}. 
  By Proposition \ref{prop:extendedgesums}, we know that $x_{1,2}+x_{1,3}=a_2$. 
  Since $x_{1,2}=16$ we obtain that $x_{1,3}=1+(a_2-17)$, so we will define
  \begin{equation*}
      d_1=a_2-17
  \end{equation*}
so that 
\begin{equation}
    x_{1,3}=1+d_1.   \label{defined1}
\end{equation}
The table in Figure \ref{d1verticalbalance} confirms that (\ref{defined1}) holds for  the fundamentally standard labelings in each 
of Figures \ref{fig:StandVertLabelsA2=18}, \ref{fig:StandVertLabelsA2=19}, \ref{fig:StandVertLabelsA2=21} 
and \ref{fig:StandVertLabelsA2=25} as well as the 4 pseudo-standard labelings in Figure \ref{pseudostandardlabeling12}. 
Also, we confirm that $d_1\in\{1,2,4,8\}$ in Figure \ref{d1verticalbalance}.
\begin{figure}
\centering
\renewcommand{\arraystretch}{1.5}
\begin{tabular}{|c|c|c|} \hline
       $a_2$  & $d_1$ & $x_{1,3}$ \\  \hline
        18  & 1 & 2  \\  \hline
       19  &  2 & 3 \\  \hline
        21  & 4 & 5 \\  \hline
        25 & 8 & 9 \\ \hline
    \end{tabular}
\caption{Confirming that $x_{1,3}=1+d_1$}\label{d1verticalbalance} 
\end{figure} 
To define $d_2$ we will set 
\begin{equation*}
d_2=x_{3,1}-1
\end{equation*}
so that $x_{3,1}=1+d_2$ as in Figure \ref{fig:SchemeTwo}. 

\begin{figure}[hbt!]
\centering
\renewcommand{\arraystretch}{1.5}
\begin{tabular}{|c|c|c|c|}\hline
  $a_2$ &  $x_{3,1}$ & $d_1$ & $d_2$ \\  \hline  
   18   &   3  &  1  & 2 \\ \hline
   18  &    5  &  1  & 4 \\  \hline
   18  &    9  &  1  & 8  \\ \hline 
\end{tabular}
\hspace{0.15in}
\renewcommand{\arraystretch}{1.5}
\begin{tabular}{|c|c|c|c|}\hline
  $a_2$ &  $x_{3,1}$ & $d_1$ & $d_2$ \\  \hline  
   19   &   2  &  2  & 1 \\ \hline
   19  &    5  &  2  & 4 \\  \hline
   19  &    9  &  2  & 8  \\ \hline 
\end{tabular}
\hspace{0.15in}
\renewcommand{\arraystretch}{1.5}
\begin{tabular}{|c|c|c|c|}\hline
  $a_2$ &  $x_{3,1}$ & $d_1$ & $d_2$ \\  \hline  
   21   &   2  &  4  & 1 \\ \hline
   21  &    3  &  4  & 2 \\  \hline
   21  &    9  &  4  & 8  \\ \hline 
\end{tabular}
\hspace{2in}\\
\hspace{2in}\\
\hspace{2in}\\
\begin{tabular}{|c|c|c|c|}\hline
  $a_2$ &  $x_{3,1}$ & $d_1$ & $d_2$ \\  \hline  
   25   &   2  &  8  & 1 \\ \hline
   25  &    3  &  8  & 2 \\  \hline
   25  &    5  &  8  & 4  \\ \hline 
\end{tabular}
\caption{Confirming the values for $d_2$.}\label{confirmingd2vertbal}
\end{figure}
As shown in Figure \ref{confirmingd2vertbal}, $d_2\in\{ 1,\, 2,\, 4,\, 8\}\setminus \{d_1\}$. 
We now define $d_3$ and $d_4$ by
\begin{equation*}
  d_3= \min\big(\{ 1,\, 2,\, 4,\, 8\}\setminus \{d_1,\, d_2\}\big)\hspace{.1in}\text{and}\hspace{.1in}d_4= \max\big(\{ 1,\, 2,\, 4,\, 8\}\setminus \{d_1,\, d_2\}\big).   
\end{equation*}
The table in Figure \ref{confirmd_3v}  verifies that $x_{2,2}=1+d_3$ and $x_{4,3}=1+d_4$ hold for the 8 fundamentally standard labelings 
in each of Figures \ref{fig:StandVertLabelsA2=18}, \ref{fig:StandVertLabelsA2=19}, \ref{fig:StandVertLabelsA2=21} and 
\ref{fig:StandVertLabelsA2=25} as well as the 4 pseudo-standard labelings in Figure \ref{pseudostandardlabeling12}.

\begin{figure}[hbt!]
\centering
\renewcommand{\arraystretch}{1.5}
\begin{tabular}{|c|c|c|c|c|c|c|c|} \hline
  $a_1$ & $a_2$  & $d_1$ & $d_2$ & $\textcolor{blue}{d_3}$ & $\textcolor{red}{d_4}$ & $x_{2,2}$  & $x_{4,3}$\\  \hline
    17  &  18  & 1 & 2 & \textcolor{blue}{4} & \textcolor{red}{8} & 5  & 9 \\  \hline
     17 &  18  &  1 & 4 & \textcolor{blue}{2} & \textcolor{red}{8} & 3 & 9 \\  \hline
     17 &    18  & 1 & 8 & \textcolor{blue}{2} & \textcolor{red}{4} & 3 & 5  \\  \hline
     17  &  19  & 2 & 1 & \textcolor{blue}{4} & \textcolor{red}{8} & 5  & 9 \\  \hline
     17 &  19  &  2 & 4 & \textcolor{blue}{1} & \textcolor{red}{8} & 2 & 9 \\  \hline
     17 &    19  & 2 & 8 & \textcolor{blue}{1} & \textcolor{red}{4} & 2 & 5  \\  \hline
    \end{tabular}\hspace{0.15in}
\begin{tabular}{|c|c|c|c|c|c|c|c|} \hline
  $a_1$ & $a_2$  & $d_1$ & $d_2$ & $\textcolor{blue}{d_3}$ & $\textcolor{red}{d_4}$ & $x_{2,2}$  & $x_{4,3}$\\  \hline
    17  &  21  & 4 & 1 & \textcolor{blue}{2} & \textcolor{red}{8} & 3  & 9 \\  \hline
     17 &  21  &  4 & 2 & \textcolor{blue}{1} & \textcolor{red}{8} & 2 & 9 \\  \hline
     17 &   21  & 4 & 8 & \textcolor{blue}{1} & \textcolor{red}{2} & 2 & 3  \\  \hline
     17  &  25  & 8 & 1 & \textcolor{blue}{2} & \textcolor{red}{4} & 3  & 5 \\  \hline
     17 &  25  &  8 & 2 & \textcolor{blue}{1} & \textcolor{red}{4} & 2 & 5 \\  \hline
     17 &    25  & 8 & 4 & \textcolor{blue}{1} & \textcolor{red}{2} & 2 & 3  \\  \hline
    \end{tabular}
\caption{Confirming that $x_{2,2}=1+d_3$ and $x_{4,3}=1+d_4$.}\label{confirmd_3v}    
\end{figure}

We will now use the expressions for $x_{1.3}$, $x_{3,1}$ and $x_{2,2}$, 
and the known values of $x_{1,1}=1$ and $x_{1,2}=16$, 
to derive expressions for
the remaining labels for these 12 labelings. 
Recall that $\textcolor{red}{1+d_1+d_2+d_3+d_4}=\textcolor{red}{16}$. 
By the vertically pairwise balanced property, 
$x_{2,1} +x_{2,2}= 17$, so that
\begin{equation}
  x_{2,1}=17-x_{2,2}=17-(1+d_3)=\textcolor{red}{16}-d_3=1+d_1+d_2+d_4.  \label{repwds1}
\end{equation}
By computations similar to (\ref{repwds1}) we obtain that $x_{3,2}=1+d_1+d_3+d_4$,\\ 
$x_{1,4}=1+d_2+d_3+d_4$, and $x_{4,4}=1+d_1+d_2+d_3$.
Since $S=34$, we have $x_{2,3}+x_{2,2}+x_{1,3}+\textcolor{blue}{x_{1,2}}=34$; 
so, by rewriting $34-\textcolor{blue}{16}=2+\textcolor{red}{16}$, we have
\begin{equation*}
  x_{2,3}=34-\textcolor{blue}{16}-x_{1,3}-x_{2,2}=2+\textcolor{red}{16}-(1+d_1)-(1+d_3)=1+d_2+d_4.  
\end{equation*}
By rewriting $S=34=2\cdot\textcolor{red}{16}+2$ we have 
\begin{equation}\label{repwds2}
\begin{aligned}
 x_{3,3}&=34-x_{3,2}-x_{2,3}-x_{2,2}\\
    &=(2\cdot\textcolor{red}{16}+2)-(1+d_1+d_2+d_4)-(1+d_2+d_4)-(1+d_3)\\
    &=1+d_1+d_2,
 \end{aligned}
\end{equation}
and in a manner similar to (\ref{repwds2}) we obtain $x_{4,2}=1+d_2+d_3$.
Finally, by computations similar to those in (\ref{repwds1}) we obtain
$ x_{2,4}=17-x_{2,3}=1+d_1+d_3$, $x_{3,4}=17-x_{3,3}=1+d_3+d_4$, and
$x_{4,1}=17-x_{4,2}=1+d_1+d_4$.
\end{proof}

\begin{theorem}\label{Numberoflabels}
In all, there are \rm 192\it\ distinct $C_4$-face-magic labelings of $\mathcal{K}_{4,4}$ up to symmetries of the Klein bottle.
\end{theorem}

\begin{proof}
    By Lemma \ref{lem:VerticalCycleOps}, an elementary vertical cycle operation $\widehat{V}_k$, for $k=2,3,4$, 
    preserves properly symmetrized labelings.
    Furthermore, applying $\widehat{V}_k$ will not produce a labeling that is Klein bottle 
    labeling equivalent to one of the standard labelings in Figures \ref{fig:StandHorzLabelsA1=9}, \ref{fig:StandHorzLabelsA1=13}, \ref{fig:StandHorzLabelsA1=15},
    \ref{fig:StandHorzLabelsA1=16}, \ref{fig:StandVertLabelsA2=18}, \ref{fig:StandVertLabelsA2=19}, \ref{fig:StandVertLabelsA2=21} and \ref{fig:StandVertLabelsA2=25}.
   Therefore, the application of $\widehat{V}_2^i\ \widehat{V}_3^j\ \widehat{V}_4^k$,
   for $i,j,k\in\{0,1\}$, to each standard labelings produces a distinct properly symmetrized labeling.

   By Theorem \ref{binaryhd_ks}, there are $4$ choices of $d_1$ and $3$ choices of $d_2$.
   Observe that $d_3$ and $d_4$ are determined from $d_1$ and $d_2$. 
   Thus, there are $4\cdot 3=12$ distinct standard horizontally pairwise balanced labelings on $\mathcal{K}_{4,4}$.  
   Each standard horizontally pairwise balanced $C_4$-face-magic labelimg is
   vertical cycle equivalent to $2\cdot2\cdot 2=8$ distinct 
   properly symmetrized labelings. 
   Consequently, there are $12\cdot 8=96$ possible properly symmetrized 
   $C_4$-face-magic labelings that arise from Figure \ref{fig:SchemeOne}.

     By applying the same counting methods to Theorem \ref{dksvertical}, 
     there are $12\cdot 8=96$ properly symmetrized vertically pairwise 
     balanced $C_4$-face-magic labelings that arise from Figure \ref{fig:SchemeTwo}.
   
    In total, there are $96+96=192$ distinct $C_4$-face-magic labelings of $\mathcal{K}_{4,4}$ up to symmetries of the Klein bottle.
\end{proof}

\section{Open Problems and Concluding Remarks}

The main results in this paper, Theorems \ref{binaryhd_ks}, \ref{dksvertical} and \ref{Numberoflabels}  are a consequence of three constraints imposed on a $C_4$-face-magic
labeling on $\mathcal{K}_{4,4}$.
The first constraint arose from the symmetries of the Klein bottle
that induced a graph isomorphism on $\mathcal{K}_{4,4}$ which led
to the concept of a properly symmetrize labeling on $\mathcal{K}_{4,4}$.
The second constraint emerged from the vertical cycle equivalence relation on
properly symmetrized $C_4$-face-magic labelings which led to the concept of a standard labeling on $\mathcal{K}_{4,4}$.
The third constraint was imposed by the edge sum invariant structure of a $C_4$-face-magic
labeling on $\mathcal{K}_{4,4}$.

The results on $C_4$-face-magic labelings on $\mathcal{K}_{4,4}$ in this paper suggests some open problems for future research.
First, we consider $C_4$-face-magic labelings on an $n\times n$ Klein bottle grid graph.

\begin{problem}
    Characterize the $C_4$-face-magic labelings
    on  $\mathcal{K}_{n,n}$ when $n$ is even.
\end{problem}

Once this problem is solved, one can consider the following more general problem.


\begin{problem}
    Characterize the $C_4$-face-magic labelings
    on  $\mathcal{K}_{m,n}$ when $m$ and $n$ are even.
\end{problem}

Clearly, the most challenging step would be to determine the permissible values for the edge sums for $\mathcal{K}_{m,n}$. Perhaps a general method for doing so could be developed with the aid of technology.

\section{Acknowledgements}
The second author was supported by a College Research Council Grant
provided by the University of Pittsburgh at Johnstown
during the research and preparation of this manuscript.

\end{document}